\documentclass[a4paper,10pt]{article}    
\usepackage[utf8]{inputenc}             
\usepackage{graphicx}
\usepackage{subcaption} 
\usepackage{float}   
\usepackage{fancybox}		  
\usepackage{makeidx} 
\usepackage{verbatim}
\usepackage{listings} 
\usepackage{fancyvrb}
\usepackage{amsfonts}
\usepackage{color}
\usepackage{colortbl}
\usepackage{color}
\usepackage{dsfont}
\usepackage{epsfig}
\usepackage{bbm}
\usepackage{tikz}
\usepackage{amssymb}
\usepackage{amsbsy}
\usepackage{amsmath}
\usepackage{enumerate}
\usepackage{stmaryrd}
\usepackage{mathtools}
  \usepackage{ulem}
\usepackage{tikz}
\usepackage{bbm}
\usepackage{fancyhdr}
\usepackage{amsthm}
\usetikzlibrary{matrix}

\newtheorem{prop}{Proposition}
\newtheorem{hyp}{Hypothesis}
\newtheorem{rem}{Remark}[section]
\newtheorem{definition}{Definition}

\usepackage[margin=1in]{geometry}
\usepackage{hyperref}
\hypersetup{                    
    colorlinks=true,                
    breaklinks=true,                
    urlcolor= blue,                 
    linkcolor= red,                
    citecolor= blue                
    }

\def\N {\mathbb{N}}

\def\R {\mathbb{R}}
\def\Z {\mathbb{Z}}

\def\T {\mathbb{T}}

\renewcommand{\P}{\mathcal{P}}
\newcommand{\x}{x^N}
\newcommand{\w}{w^N}

\newcommand{\Rd}{\R^d}
\newcommand{\setn}{\{1,\cdots,N\}}
\newcommand{\tmu}{\tilde{\mu}}
\newcommand{\bx}{\bar{x}}
\newcommand{\setd}{\{1,\cdots,d\}}
\newcommand{\bmu}{\bar{\mu}}
\newcommand{\dbl}{d_\mathrm{BL}}
\newcommand{\Lip}{\mathrm{Lip}}
\newcommand{\dsp}{\displaystyle}
\newcommand{\BL}{\mathrm{BL}}

\newcommand{\tZ}{\tilde{Z}}
\newcommand{\PP}{\mathbb{P}}

\newcommand{\PR}{\P(\R^d)}
\newcommand{\se}{\zeta}
\newcommand{\elts}{\{1,\cdots,N\}}

\newcommand{\txn}{x_N}
\newcommand{\tmn}{m_N}
\DeclarePairedDelimiter\floor{\lfloor}{\rfloor}
\newcommand*\circled[1]{\tikz[baseline=(char.base)]{
            \node[shape=circle,draw,inner sep=2pt,solid] (char) {#1};}}

\begin{document}

\title{Large-population limits of non-exchangeable particle systems}

\author{
Nathalie Ayi\thanks{Sorbonne Universit\'e, CNRS, Universit\'e Paris Cit\'e, Inria, Laboratoire Jacques-Louis Lions (LJLL), F-75005 Paris, France} 
 \and
Nastassia Pouradier Duteil\thanks{Sorbonne Universit\'e, Inria, CNRS, Universit\'e Paris Cit\'e, Laboratoire Jacques-Louis Lions (LJLL), F-75005 Paris, France} }


\maketitle

\abstract{ 
A particle system is said to be non-exchangeable if two particles cannot be exchanged without modifying the overall dynamics. 
Because of this property, the classical mean-field approach fails to provide a limit equation when the number of particles tends to infinity.
In this review, we present novel approaches for the large-population limit of non-exchangeable particle systems,
based on the idea of keeping track of the identities of the particles. 
These can be classified in two categories. 
 The non-exchangeable mean-field limit describes the evolution of the particle density on the product space of particle positions and labels.
Instead, the continuum limit allows to obtain an equation for the evolution of each particle's position as a function of its (continuous) label.
We expose each of these approaches in the frameworks of static and adaptive networks.
}

\tableofcontents




\section{Introduction}\label{sec:Intro}

\subsection{Setting}
Interacting particle systems refer to a large class of coupled differential equations modeling populations of interacting agents (``particles'') susceptible to exhibit global organizational patterns without any centralized intelligence. 
They are used by various scientific communities to model many different phenomena, such as opinion formation \cite{HK}, organization within animal groups \cite{Aoki82, Ballerini08, Lopez12}, synchronization of coupled oscillators \cite{Kuramoto75} and biological tissue formation \cite{BDPZ20}.
These models have also been of great interest to the mathematical community. In particular, the study of  their long-time behavior can reveal remarkable \textit{global} patterns, although the particles interact only \textit{locally} (see for instance \cite{BonnetPouradierDuteilSigalotti21,BST22,HaKimZhang18,HaLiu09,JadbabaieMoteeBarahona04}). Another direction of work consists of studying the control or optimal control of such systems, to drive the population to a target configuration (see \cite{BulloCortezMartinez09} and references within). A third area of interest concerns their so-called \textit{large-population limit}, which is the focus of the present review.

The need for large-population limits of interacting particle systems becomes evident when one considers the size of some of the populations of interest. For instance, starlings are known to form flocks of several thousands of individuals \cite{Ballerini08}, and neural networks in mammals contain up to $10^{11}$ neurons \cite{JabinPoyatoSoler21}. 
In this context, simulating such large populations via a system of $N$ coupled equations (with up to $N^2$ pairwise interactions) rapidly becomes computationally too expensive as the number of particles $N$ increases. 
To resolve this issue, a crucial idea introduced in the pioneering works \cite{BraunHepp77, Dobrushin79, Golse16, JabinWang18, NeunzertWick74, Serfaty20, Sznitman91} consists of focusing on the evolution of the population density, instead of the individual trajectories.

\paragraph{Mean-field limit of exchangeable particle systems.} Consider a system of $N$ interacting particles characterized by their positions $(\x_i)_{i\in\setn}\in C([0,T];(\R^d)^N)$, whose evolution is given by the general system of coupled differential equations:
\begin{equation}\label{eq:micro_exchangeable}
    \begin{cases}
\displaystyle \frac{d}{dt} \x_i(t) = \frac{1}{N}  \sum_{j=1}^N \phi( \x_i(t),\x_j(t)) \quad \text{ for all } i\in\setn, \quad t\in [0,T] \\
\x_i(0) = x_i^{N,0}\text{ for all } i\in\setn,
\end{cases}
\end{equation}
where
\begin{itemize}
    \item $\setn$ denotes the set of labels of the interacting particles, which represents their individual identities; 
    \item $(\x_i)_{i\in\setn}\in C([0,T];(\R^d)^N)$ denotes the state of the interacting particles, and can for instance represent positions \cite{FaureMaury15}, opinions \cite{HK}, velocities \cite{CS07}, phases \cite{Kuramoto75}, angles \cite{Vicsek95}. The dimension $d$ of the state space is then determined by the model;
    \item $\phi:\R^d\times\R^d\rightarrow\R^d$ is the so-called \textit{interaction function}, encoding the pairwise interaction between any two positions $x_i$ and $x_j$.
\end{itemize}

As $N$ tends to infinity, one can show that the system \eqref{eq:micro_exchangeable} is well approximated by its so-called \textit{mean-field limit} $\mu\in C([0,T];\P(\R^d))$ which satisfies the following Vlasov equation:
\begin{equation}\label{eq:meanfield_exchangeable}
\begin{cases}
\displaystyle \partial_t \mu_t(x) + \nabla_x\cdot \left( \left( \int_{\R^d} \phi(x,y) d\mu_t(y) \right) \mu_t(x) \right) = 0,\\
\mu_{t=0} = \mu_0.
\end{cases} 
\end{equation}
For reasons that will become clear, in this review we will refer to \eqref{eq:meanfield_exchangeable} as the \textbf{exchangeable mean-field limit} of the particle system.
The link between the particle system \eqref{eq:micro_exchangeable} and its mean-field limit equation \eqref{eq:meanfield_exchangeable} becomes clear when introducing the so-called \textit{empirical measure} $\mu^N$ constructed from the solution $(\x_i)_{i\in\setn}$ to the microscopic system \eqref{eq:micro_exchangeable}:
\begin{equation}\label{eq:empirical_measure_exchangeable_intro}
    \mu^N_t := \frac{1}{N} \sum_{i=1}^N \delta_{\x_i(t)},
\end{equation}
and which satisfies the mean-field equation \eqref{eq:meanfield_exchangeable}. Convergence of the solution $(\x_i)_{i\in\setn}$ of the microscopic system \eqref{eq:micro_exchangeable} towards the solution $\mu$ to the mean-field equation \eqref{eq:meanfield_exchangeable} can then be obtained as the result of a stability argument for the solution to \eqref{eq:meanfield_exchangeable}, of the form:
\[
W_1(\mu_t,\mu_t^N) \leq C(T) W_1(\mu_0,\mu_0^N),
\]
where $W_1$ is the $1$-Wasserstein (also known as Rubinstein-Kantorovich) distance \cite{Dobrushin79}.

\paragraph{Exchangeable vs. non-exchangeable particle systems.}
Interacting particle systems can be broadly classified into two categories: \textit{exchangeable} (or indistinguishable) and \textit{non-exchangeable} (or non-indistinguishable). A particle system is said to be exchangeable if any two particles can be exchanged without modifying the dynamics of the other particles. More precisely, 
\begin{definition}\label{def:exchangeability}
     Let $(x_i)_{i\in\setn}\in C([0,T];(\R^d)^N)$ denote the trajectories of $N$ interacting particles
     satisfying 
    \[
    \begin{cases}
    \dsp \frac{d}{dt} \x_i(t) = F_i(t,\x(t)), \quad i\in \setn, \quad t\in [0,T],\\
    \x_i(0) = x_i^{N,0}, \quad i\in\setn,
    \end{cases}
    \] 
    and let $\x_i(t) := \Psi_t\# x_i^{N,0}$ denote their push-forward by the flow of $F$.
    The particle system is said to be \emph{exchangeable} (or, equivalently, \emph{indistinguishable}) if, for any permutation function $\sigma:\setn\rightarrow\setn$ of the sets of particle labels, it holds
    \[
    \begin{split}
        \forall i\in \setn, \; y^N_i(0)&=\x_{\sigma(i)}(0) \\
    &  \Rightarrow\quad \forall i\in \setn, \; \forall t\in [0,T], \; (\Psi_t\#y^{N,0}_i)(t)= (\Psi_t\#x^{N,0}_{\sigma(i)})(t).
    \end{split}
    \]
\end{definition}

In other words, an interacting particle system is exchangeable if one can relabel the particles without modifying the dynamics. Trivially, the particle system \eqref{eq:micro_exchangeable} is exchangeable, since the right-hand side does not depend explicitly on the particle labels. 

Exchangeable particle systems provide a good modeling framework for many applications in which the \textit{labels}, which represent the \textit{identities} of the particles, do not influence the dynamics. In this case, considering all particles to be identical is a good approximation, and has been validated experimentally \cite{Aoki82, Ballerini08, Lopez12}.

However, in other cases, the particles' labels play a significant role, which requires the use of models for non-exchangeable particles.
For instance, the Kuramoto model, used to describe the evolution of coupled oscillators' phases $(\x_i)_{i\in\elts}\in C([0,T];(\T^1)^N)$, can be written, for some $k \in \{1,\dots,N\}$, as
\begin{equation}\label{eq:Kuramoto_intro}
    \frac{d}{dt}\x_i(t) = u^N_i + \frac{C}{N} \sum_{j=1}^{N} \sin(\x_j(t)-\x_i(t)).
\end{equation}
Notice that in \eqref{eq:Kuramoto_intro}, the evolution of each oscillator's phase $\x_i$ depends on an intrinsic frequency $u^N_i$, so this system does not belong to the class of non-exchangeable particle systems.

\paragraph{Why the classical mean-field limit fails for non-exchangeable particle systems.}
In this review, we will focus on \textit{non-exchangeable} particle systems of the form: 
\begin{equation}\label{eq:micro_nonexchangeable}
\begin{cases}
\displaystyle \frac{d}{dt} \x_i(t) = \frac{1}{N}  \sum_{j=1}^N \w_{ij} \phi(\x_i(t),\x_j(t)) \quad \text{ for all } i\in\setn, \;t\in [0,T], \\
\x_i(0) = x_i^{N,0}.
\end{cases}
\end{equation}
In system \eqref{eq:micro_nonexchangeable}, the effect of the particles' labels $\setn$ on the dynamics is decoupled from that of the particles' positions $(\x_i)_{i\in\setn}$, as it is introduced as a family of multiplying weights $(\w_{ij})_{i,j\in\setn}\in\R^N$.
Particle systems of the form \eqref{eq:micro_nonexchangeable} can be seen as posed on an underlying weighted graph, in which the set of nodes corresponds to the set of labels $\setn$ and to each edge $(i,j)$ is attributed a weight $\w_{ij}\in\R$.

Finding a good approximation of non-exchangeable particle systems \eqref{eq:micro_nonexchangeable} when $N$ is large has been the subject of many recent works \cite{AyiPouradierDuteil23, ChibaMedvedev19, GkogkasKuehn22, KaliuzhnyiMedvedev18, KuehnXu22, JabinWang18, Medvedev14, MedvedevR, PT23}.
One main difficulty comes from the fact that the classical mean-field approach is no longer applicable. Indeed, studying the population density instead of the individual particles' trajectories entails an irreversible information loss, as one loses track of the particles' labels. 
In particular, notice that the empirical measure is blind to any permutation $\sigma:\setn\rightarrow\setn$ of the sets of indices, since 
\[\mu^N_t := \frac{1}{N} \sum_{i=1}^N \delta_{\x_i(t)} =  \frac{1}{N} \sum_{i=1}^N \delta_{\x_{\sigma(i)}(t)}.\]
Due to the definition of non-exchangeable particle systems, one cannot hope to be able to capture the dynamics of system \eqref{eq:micro_nonexchangeable} with a mean-field equation of the type \eqref{eq:meanfield_exchangeable}.

\paragraph{New results in Limits of Graphs.} Non-exchangeable particle systems of the form \eqref{eq:micro_nonexchangeable} can be seen as posed on an underlying graph. For instance, in \cite{WileyStrogatzGirvan06}, a modified version of the Kuramoto model is introduced, in which the coupled oscillators are positioned in a one-dimensional ring, so that each oscillator interacts only with $k$ oscillators on each side (for $k<\frac{N}{2}$):
\begin{equation}\label{eq:Kuramoto_intro_2}
    \frac{d}{dt}\x_i(t) = u + \frac{C}{N} \sum_{j=i-k}^{i+k} \sin(\x_j(t)-\x_i(t)).
\end{equation}
System \eqref{eq:Kuramoto_intro_2} can be interpreted as a system of differential equations on the graph $G_N = \langle V(G_N), E(G_N)\rangle$ whose vertex set is $V(G_N) = \{1, \dots,N \}$ and edge set is 
$$E(G_N) = \{(i,j) \in \{1,\dots,N\} : 1 \leq \text{dist}(i,j) \leq k \} \, \text{where dist}(i,j) = \min\{|i-j|, N- |i-j| \}.$$
The graph $G_{10}$ and its adjacency matrix $(\w_{ij})_{i,j\in\setn}$ are represented in Figure \ref{fig:graph}.
 \begin{figure}
     \centering
     \begin{tabular}{ccc}
     \includegraphics[width = 0.3\textwidth]{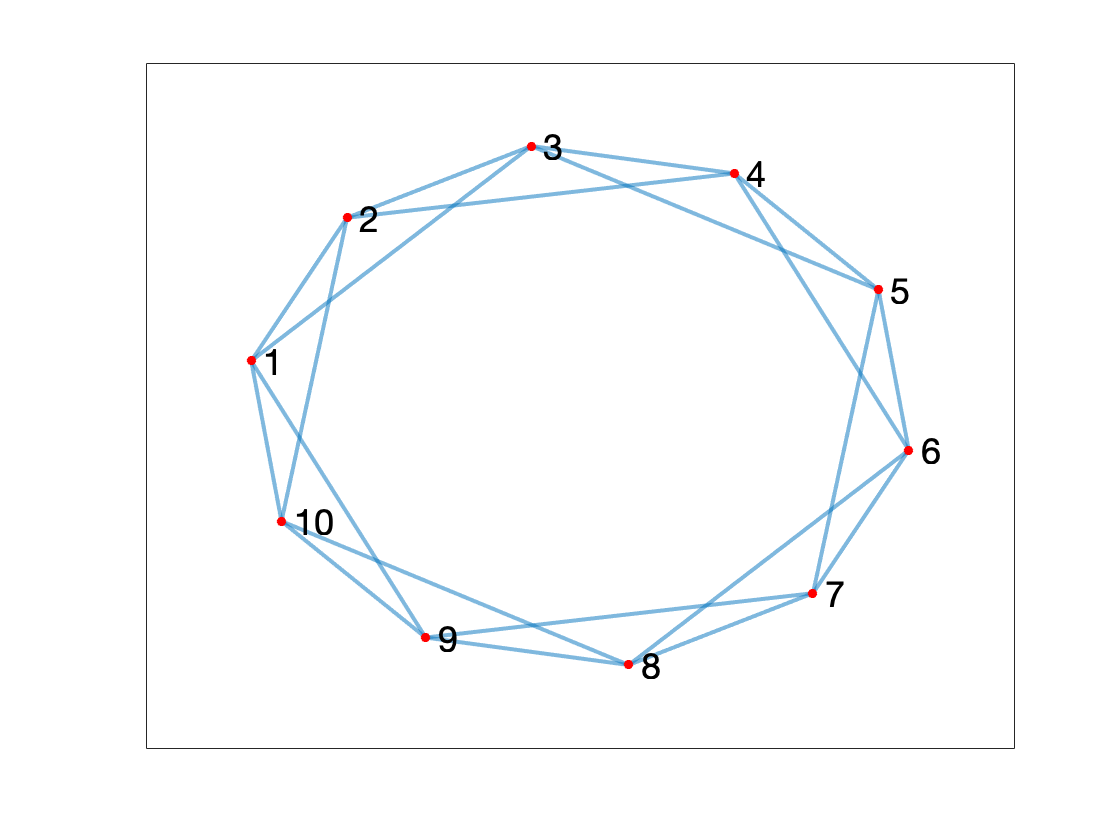}
     &
     \includegraphics[width = 0.3\textwidth]{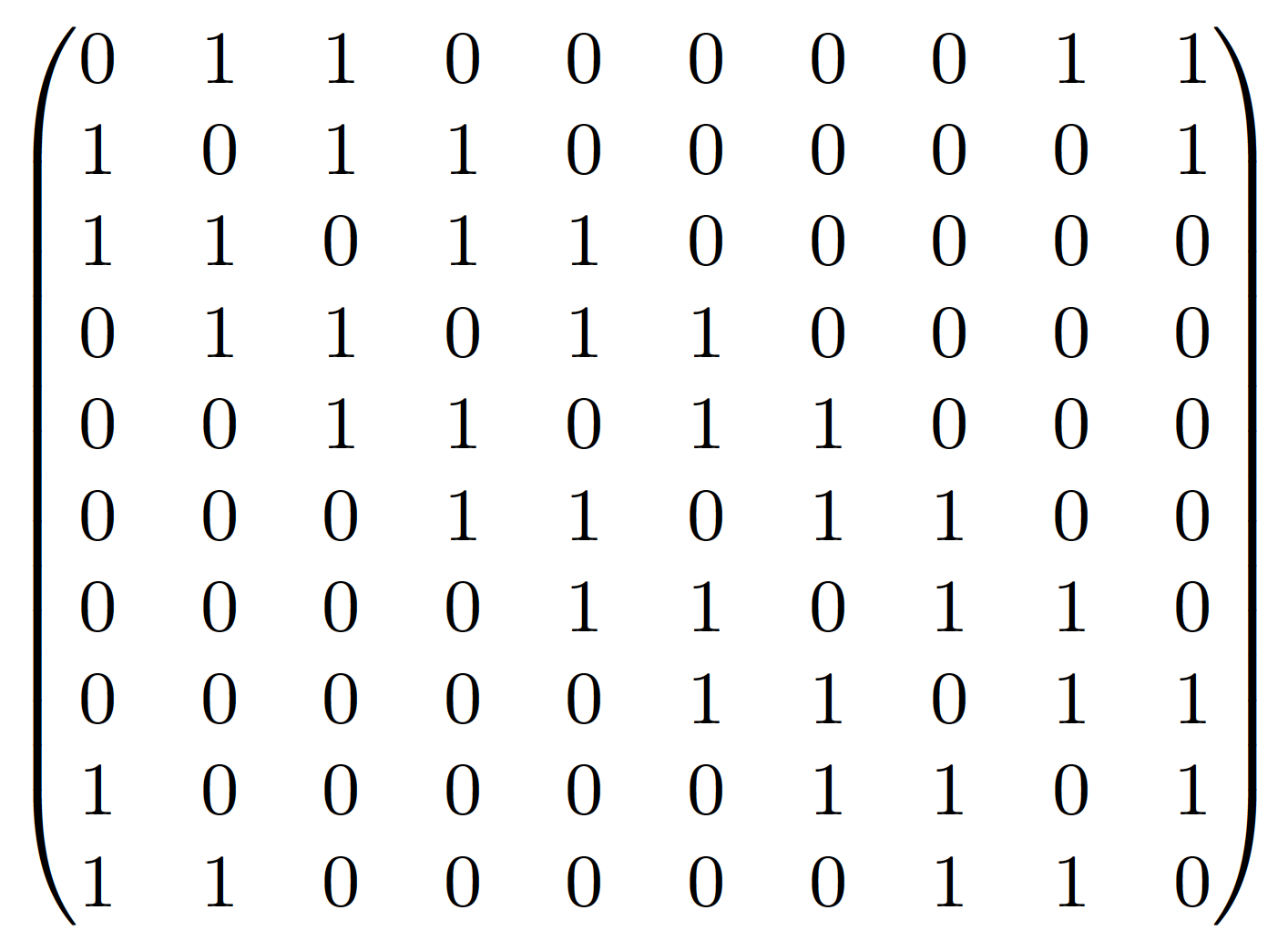}
     \end{tabular}
\caption{Graph and adjacency matrix associated with System \eqref{eq:Kuramoto_intro_2}, for $N=10$ and $k = 2$.}
\label{fig:graph}    
 \end{figure}

 Notice that the adjacency matrix of $G_N$ belongs to $M_N(\R)$, whose dimension changes with each value of $N$.
One can instead introduce the piecewise-constant function $w_{G_N}: I^2 \to \{0,1\}$, where $I:=[0,1]$, such that 
\begin{equation*}
w_{G_N}(\xi,\zeta) = \left\{\begin{array}{l}
1 \qquad \text{~if~}  (\xi,\zeta)  \in \left[\frac{i-1}{N}, \frac{i}{N}\right) \times  \left[\frac{j-1}{N}, \frac{j}{N}\right) \text{~and~}  (i,j) \in E(G_N),\\
0 \qquad \text{~otherwise.}
\end{array} \right.
\end{equation*}
The plot of $w_{G_N}$'s support is nothing else than a pixel representation of the adjacency matrix of $G_N$ (see Figure \ref{fig:pixel_graphon}). Moreover, with this new characterization, the space $L^\infty(I^2)$ to which $w_{G_N}$ belongs no longer varies with $N$, which allows to consider its limit in the same space.
 In example \eqref{eq:Kuramoto_intro_2}, if $k$ is proportional to $N$, one can show that $w_{G_N}$ converges as $N$ goes to infinity to a $\{0,1\}$-valued function denoted $w$. This function $w$ is called a \textit{graphon} in Graph Theory. This crucial object appears when considering limits of \textit{dense} sequences of convergent graphs $(G_N)_{N\geq 1}$ (see \cite{LovaszSzegedy06}), that is sequences in which $|E(G_N)|= O (|V(G_N)|^2)$ as $N$ goes to infinity. 

 \begin{rem}
     Note the graph $G_N$ associated with example \eqref{eq:Kuramoto_intro_2} is an undirected, unweighted graph (i.e. the associated adjacency matrix contains either zeros or ones). The theory applies similarly to directed and weighted graphs, for which the adjacency matrix $(\w_{ij})_{i,j\in\setn}$ takes values in $\R$ and is not necessarily symmetric.
 \end{rem}

The limit of the particle system \eqref{eq:micro_nonexchangeable} when $N$
 goes to infinity then naturally reveals objects from Graph Theory, which provides powerful tools to tackle the question of the large population limit (see \cite{BackhauszSzegedy22,LovaszSzegedy06} for example). For instance, the convergence of a graph sequence is linked to the convergence of $w_{G_N}$ to $w$ in the so-called \textit{cut-norm}, which is defined, for $w \in L^1(I^2)$, as
 \begin{equation}\label{eq:cut-norm}
     \|w\|_\square := \sup_{\substack{S,T \text{~measurable}\\\text{subsets of~} I}} \left| \int_{S \times T} w(\xi,\zeta) d\xi d \zeta\right|
 \end{equation}
 One property of this norm is that 
 \begin{equation}
 \|w\|_\square \leq \|w\|_{L^1(I^2)}.
 \end{equation}
Thus, convergence in $L^1$-norm implies convergence of the graph sequence $(G_N)_{N\geq 1}$. This explains why the $L^p$- setting is adopted in many frameworks (as will be seen in Section \ref{sec:MFLdense} and \ref{sec:Graphs_GL}). One of the drawbacks of this setting is that it is only valid for \textit{dense} graphs.

 However, recent results in \cite{BackhauszSzegedy22} allow to revisit this question in order to address intermediate densities or sparse graphs. In this article, Backhaus and Szegedy provide a general framework unifying dense and sparse graph limit theory. Their approach relies on the fact that graphs can actually be represented as operators called \textit{graphops}. More precisely, a graphop is  a bounded self-adjoint and positivity preserving operator $A:L^\infty(I) \to L^1(I)$. 
 Graphops can be seen as a generalization of graphons by defining for each graphon $w\in L^\infty(I^2)$ an associated graphop
 $$A_wf(\zeta) := \int_I w(\xi,\zeta) f(\zeta) d\zeta.$$ 
 Note that there exists alternative objects more general than graphons but less than graphops such as digraph measures \cite{KuehnXu22} or extended graphons \cite{JabinPoyatoSoler21}. We will introduce them when presenting the associated large population results (Section \ref{sec:MFLsparse}). 
 \begin{figure}
     \centering
     \begin{tabular}{ccc}
\includegraphics[width = 0.3\textwidth]{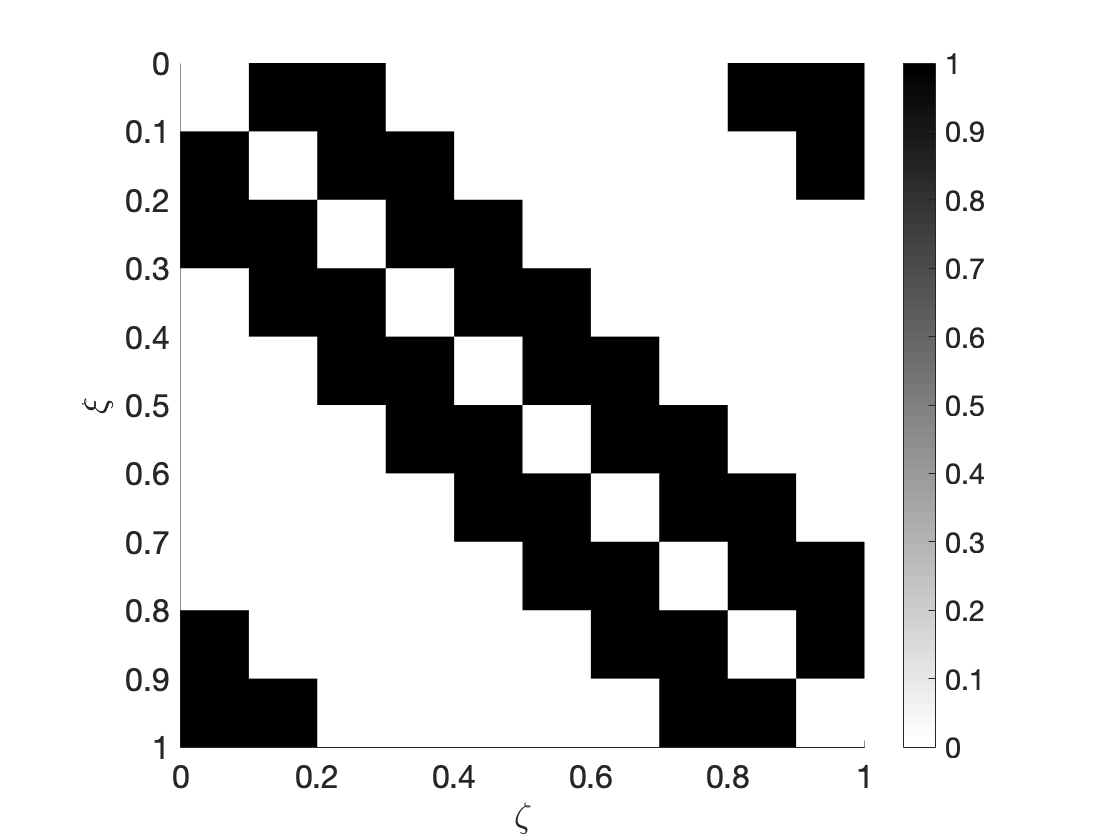} &
\includegraphics[width = 0.3\textwidth]{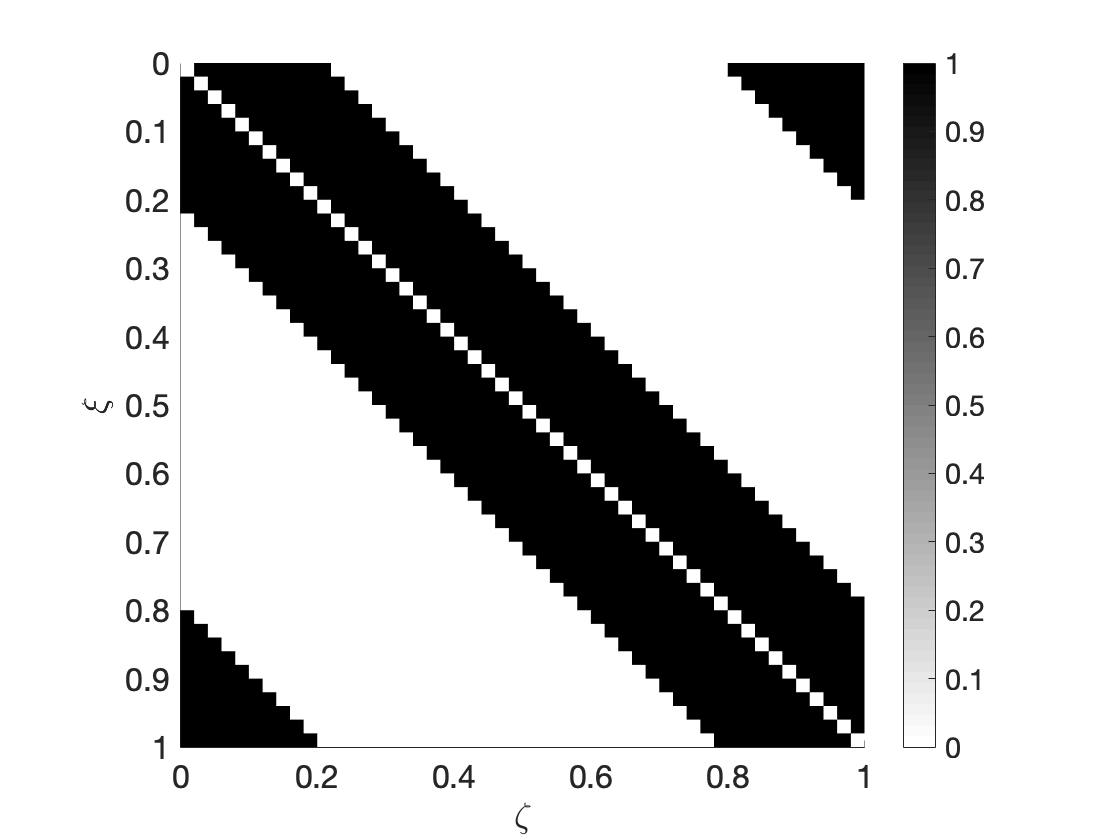} &
\includegraphics[width = 0.3\textwidth]{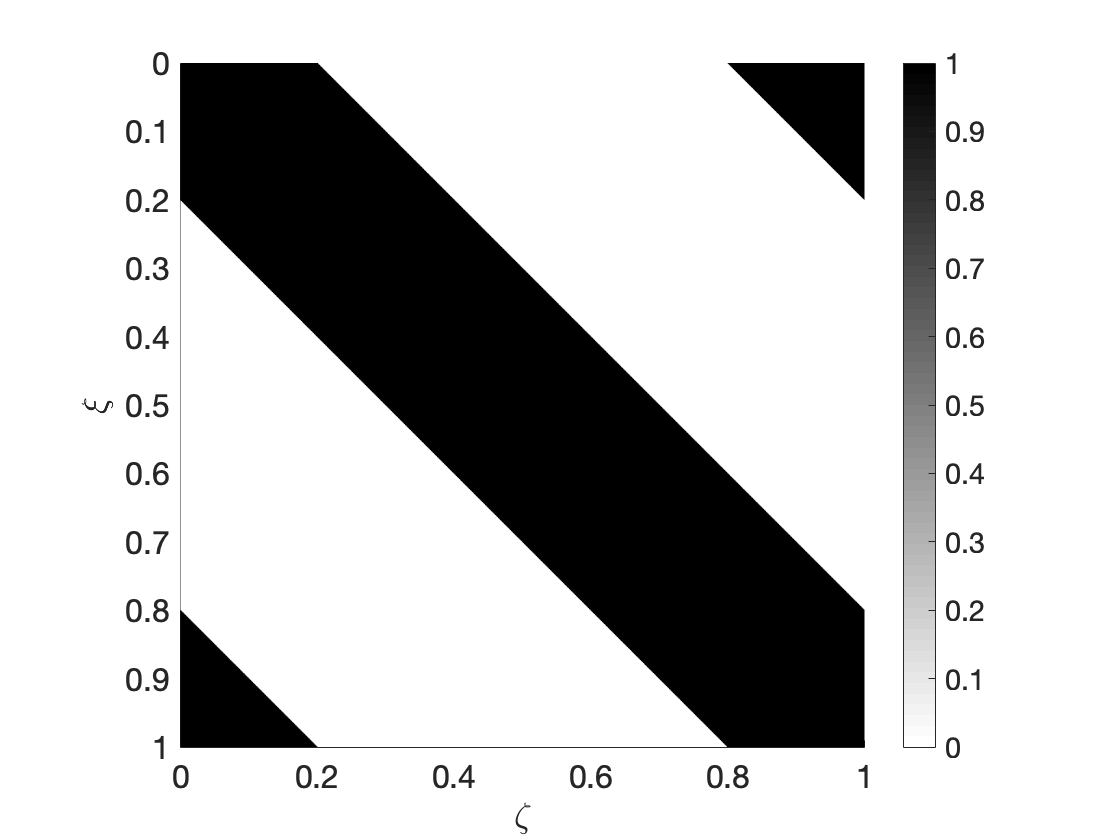} 
\end{tabular}
     \caption{Left and center: Pixel matrices of the graphs associated with \eqref{eq:Kuramoto_intro} for $N=10$ and $N=50$, with $k = \frac{N}{5}$. Right: Plot of the limit graphon $w$. }
     \label{fig:pixel_graphon}
 \end{figure}
 
\paragraph{The non-exchangeable mean-field limit and the continuum limit.}
If the interaction weights $(\w_{ij})_{i,j\in\setn}$ converge as $N$ tends to infinity to a limit object, which can be either a graphon $w\in L^\infty([0,1]^2)$ or a more general object (such as an extended graphon, a graphop, or a digraph measure),
then the microscopic system can be shown to converge weakly towards a measure $\mu\in C([0,T],\P([0,1]\times\R^d))$, solution to a Vlasov equation of the form:
\begin{equation}\label{eq:meanfield_nonexchangeable}
\begin{cases}
\dsp \partial_t \mu_t^\xi(x) + \nabla_x\cdot \left( \left( \int_{[0,1]} \int_{\R^d} w(\xi,\zeta) \phi(x,y) d\mu_t^\zeta(y) d\zeta \right) \mu_t^\xi(x) \right) = 0,\\
\mu_{t=0} = \mu_0
\end{cases} 
\end{equation}
in the graphon case. 
To differentiate this limit from the classical mean-field limit \eqref{eq:meanfield_exchangeable}, we will refer to Equation \eqref{eq:meanfield_nonexchangeable} as the \textbf{non-exchangeable mean-field limit} of System \eqref{eq:micro_nonexchangeable}.
Here, the variable $x\in \R^d$, as in the microscopic model \eqref{eq:micro_nonexchangeable}, represents the position (or phase, opinion, etc.), and the newly introduced variable $\xi\in [0,1]$ is a continuous representation of the particles' labels, or identities.
Notice in particular the asymmetric roles of of the two variables $x$ and $\xi$. 
The limit equation is a transport equation, in which the probability measure $\mu$ is transported only in the direction of the variable $x$, whereas the variable $\xi$ plays the role of a structure variable.
Due to its similarity with the previously mentioned mean-field limit, we will refer to Equation \eqref{eq:meanfield_nonexchangeable} as the \textit{non-exchangeable mean-field limit} of system \eqref{eq:micro_nonexchangeable}.
The non-exchangeable mean-field limit has been derived in various frameworks in \cite{ChibaMedvedev19, GkogkasKuehn22, KaliuzhnyiMedvedev18, KuehnXu22, JabinWang18, PT23}, presented in Section \ref{sec:Graphs_MFL}.

As shown in \cite{AyiPouradierDuteil23, Medvedev14, MedvedevR, PT23}, when the interaction weights $(\w_{ij})_{i,j\in\setn}$ converge as $N$ tends to infinity to a graphon $w\in L^\infty([0,1]^2)$,
the microscopic system \eqref{eq:micro_nonexchangeable} can also be shown to converge pointwise 
 towards the solution $x\in C([0,T]; L^\infty(\R^d))$ to an integro-differential Euler-type equation (also denoted \textit{nonlinear heat equation} in \cite{Medvedev14}): 
\begin{equation}\label{eq:GL_nonexchangeable}
\begin{cases}
\dsp \partial_t x(t,\xi) = \int_{[0,1]} w(\xi,\zeta) \phi(x(t,\xi),x(t,\zeta)) d\zeta \\
x(0,\cdot) = x_0.
\end{cases}
\end{equation}
This approach is referred to as \textit{continuum limit} or \textit{graph limit}, and is presented in Section \ref{sec:Graphs_GL}.

The relations between the two limit equations \eqref{eq:meanfield_nonexchangeable} and \eqref{eq:GL_nonexchangeable} are presented in Section \ref{sec:Link}.

\paragraph{Adaptive dynamical networks.}
Adaptive dynamical networks represent a broad class of interconnected dynamical systems, very useful to cover a wide range of real-life applications. Their main feature is that  the connectivity of the network evolves over time and that this evolution can depend on the  states of the system itself. For instance, in the context of opinion dynamics, nodes symbolize individual agents, while links mirror the myriad connections that we maintain in our social spheres, be it with friends, family, or colleagues. In that framework, the adaptive nature is to be understood as follows: not only are relationships likely to influence our opinions, but our opinions also exert a reciprocal effect, inducing alterations in the network structure (our relationships). An illustrative instance is found in the adaptive voter model (see for instance \cite{Zschaler_2012}), where agents update their opinions and connections. With a certain probability, each agent may adopt the viewpoint of interacting agents or, alternatively, shift their connections towards those who share more similar opinions. These intricate models offer a nuanced representation of reality, acknowledging that in numerous scenarios, networks are far from static. This not only holds true  as mentioned above for social interactions \cite{AyiPouradierDuteil21,BenPorat2023,McQuadePiccoliPouradierDuteil19,nugent2023evolving,PiccoliPouradierDuteil21,PouradierDuteil21,porat2023graph}, but extends to diverse domains such as neural and neuronal networks  \cite{Popovych,Rohr}, machine learning applications  \cite{LeCun,Rumelhart}, power grid models \cite{Nishikawa_2015}, and beyond. For a comprehensive exploration of various domains and examples, we direct readers to the review~\cite{rev_kuehn}.

As in \cite{rev_kuehn}, we will say that a network is \textit{adaptive} if the evolution of the edge $(i,j)$ explicitly depends on the states of the nodes $i$ and $j$. A counter-example is the blinking system in \cite{AyiPouradierDuteil23} for which every edge of the graph is reset at regular time intervals, independently of the states of the nodes at this time. To summarize, as mentioned in \cite{rev_kuehn}, there are two main classes for adaptive networks: 
\begin{itemize}
\item \textit{event-based adaptation}, where the network structure changes at certain discrete points in time and the triggers for the changes depend on the system itself,
\item \textit{continuous adaptation}.
\end{itemize}
In this review, we will exclusively focus on the second class. It can be presented in its general form as 
\begin{equation}
\begin{cases}
\displaystyle \frac{d}{dt} \x_i(t) = f_i(\x_i(t),t) +  \sum_{j=1}^N \w_{ij}(t) \phi \left(\x_i(t),\x_j(t),t\right) \quad \text{ for all } i\in\setn \\
\displaystyle \frac{d}{dt} w_{ij}^N(t) = h_{ij}(w^N(t),x^N(t),t)
\end{cases}
\end{equation}
where $\displaystyle w^N:=(w^N_{ij})_{i,j \in \setn}$, with $w_{ij}^N$ being the (now time-evolving) weight of the connection from node $j$ to $i$, and $x^N=(x_1^N, \dots, x_N^N)$.\\
The question of the large population limit, whether through the mean-field approach or the continuum limit one, is extremely challenging in its generality. Thus, to this date, only two very specific frameworks have successfully addressed this problem. The first one is the setting of Kuramoto-type models on $N$ oscillators \cite{gkogkas2023}. The second one can be seen as a variant of the Hegselmann-Krause dynamics \cite{HK} where, additionally to the opinions, we also are interested in the evolving-in-time weights of agents which represent their charisma, their popularity \cite{AyiPouradierDuteil21}.

We will present mean-field limits for those models in Section~\ref{mfl_adaptive} and continuum limits in Section~\ref{sec:GL_adaptive}. As in the static case, we will exhibit the links between the different limit equations in Section~\ref{Links2}.


\subsection{Notations and preliminary remarks}

In the last ten years, the surge of interest in large-population limits for non-exchangeable particle systems has led to a wealth of publications.
Consequently, due to the rapid growth of the field and to the variety of models considered, no unifying set of notations has emerged within the community.
In this review, we have decided to use the following notations: 
\begin{itemize}
    \item $x$ will denote a particle's state (which can represent an opinion, position, velocity or phase, as mentioned above). Note that it is also denoted by $x$ in \cite{AyiPouradierDuteil21,JabinPoyatoSoler21}, but it is denoted by $\xi$ in \cite{PT23}, by $u$ in \cite{AyiPouradierDuteil23, ChibaMedvedev19,GkogkasKuehn22,KaliuzhnyiMedvedev18,Medvedev14, MedvedevR}, by $\phi$ in \cite{KuehnXu22} and by $\theta$ in \cite{BetCoppiniNardi23}.
    \item $\Omega$ will denote the state space in which the particles evolve. We can distinguish between two main classes of results. Articles focusing on the Kuramoto model \eqref{eq:Kuramoto_intro} consider that $x$ represents an oscillator phase, and thus take the state space to be the torus $\Omega= \T := \R \slash (2 \pi \Z)$  
    \cite{ChibaMedvedev19,GkogkasKuehn22,gkogkas2023,gkogkas2023mean,KaliuzhnyiMedvedev18}. On the other hand, when the state is assumed to be a position, the state space is assumed to be $\Omega=\R^d$, for instance in \cite{AyiPouradierDuteil21,JabinPoyatoSoler21,PT23}. For presentation simplicity or technical reasons, one can also find $\Omega=\R$ in \cite{AyiPouradierDuteil23,Medvedev14,MedvedevR}.
    \item $\xi$ will denote a particle's label, or identity. It is also denoted by $\xi$ in \cite{JabinPoyatoSoler21}, but is denoted by $s$ in \cite{AyiPouradierDuteil21,BiccariKoZuazua19} and by $x$ in \cite{BetCoppiniNardi23,GkogkasKuehn22,KuehnXu22,Medvedev14, MedvedevR, PT23}.
    \item $I$ will denote the space to which the (continuous) label $\xi$ belongs. Most often, the discrete label $i \in \setn$ will be mapped to the continuous label $\xi$ by the transformation $i\mapsto\frac{i}{N}$, so that $I=[0,1]$ \cite{AyiPouradierDuteil21,AyiPouradierDuteil23,ChibaMedvedev19,KaliuzhnyiMedvedev18,Medvedev14,MedvedevR}. In other works, the label set is taken to be a more general multi-dimensional set, denoted $X$ \cite{KuehnXu22} or $\Omega$ \cite{GkogkasKuehn22,PT23}. To keep the presentation as clear as possible, we chose to state all results in the framework $I=[0,1]$, but will mention the possible extensions to more general sets when applicable.
    \item $\phi$ will denote the interaction function, as in \cite{AyiPouradierDuteil21,JabinPoyatoSoler21}. It is also denoted by $g$ in \cite{KuehnXu22}, and by $D$ in \cite{GkogkasKuehn22,KaliuzhnyiMedvedev18,Medvedev14,MedvedevR}.
    Note that in many applications (such as the Kuramoto model mentioned above, see Equation \eqref{eq:Kuramoto_intro}), $\phi:(x,y)\mapsto \tilde\phi(y-x)$ for some function $\tilde\phi$.
    \item $(\w_{ij})_{i,j\in\setn}$ will denote the interaction weights as in \cite{JabinPoyatoSoler21,Medvedev14}. The weight $\w_{ij}$ is denoted by $W_{N,ij}$ in \cite{KaliuzhnyiMedvedev18}, by $a_{ij}$ in \cite{KuehnXu22}, by $A_{ij}$ in \cite{GkogkasKuehn22}, by $\xi^N_{ij}$ in \cite{MedvedevR}.
\end{itemize}

In order to keep our presentation as clear as possible given the variety of existing frameworks, we have made the choice to present some results in a simplified formalism. We encourage the reader to consult the more general statements in the corresponding original articles.


\section{Large-population limits of non-exchangeable particle systems on static graphs}

Let $T>0$ and $x_i^{N,0}\in\Omega^N$. Let $(\x_i)_{i\in\setn}\in C([0,T];\Omega^N)$ represent a system of non-exchangeable interacting particles satisfying the system of coupled equations
\begin{equation}\label{eq:micro1}
\begin{cases}
\displaystyle \frac{d}{dt} \x_i(t) = \frac{1}{N}  \sum_{j=1}^N \w_{ij} \phi(\x_j(t)- \x_i(t)) \quad \text{ for all } i\in\setn, \; t\in [0,T] \\
\x_i(0) = x_i^{N,0}\quad \text{ for all } i\in\setn.
\end{cases}
\end{equation}
\begin{itemize}
\item The set $\setn$ represents the set of \textit{labels} of the interacting particles, i.e. their identities. Two distinct particles cannot have the same label.
\item For each $t\in [0,T]$, $\x_i(t)\in \Omega$ represents the state of the particle with label $i$ at time $t$. Depending on the model, it might represent an opinion, a position (in which case $\Omega=\R^d$), a phase ($\Omega=\T$) etc. For simplicity, from here onwards, we will refer to $\x_i(t)$ as the particle's \textit{position}.
\item The interaction function $\phi\in \Lip(\Omega^2)$ models the spatial interaction between any two particles. It will be considered to be a Lipschitz function.
\item $(\w_{ij})_{i,j\in\setn}$ are interactions weights, modeling an underlying interaction network. We suppose that for every $(i,j)\in\N^2$, $\w_{ij}\in\R$. Depending on the framework, it will also sometimes be assumed that the matrix $(\w_{ij})_{i,j\in\setn}$ is symmetric, corresponding to an undirected graph, and that $\w_{ij}\in\R_+$. These assumptions will be made clear when necessary.
\end{itemize}

Under these hypotheses, there exists a unique solution $(\x_i)_{i\in\setn}\in C([0,T]; \Omega^N)$ to \eqref{eq:micro1} (see \cite{PT23} for the proof under more general hypotheses).

\subsection{Mean-field limit} \label{sec:Graphs_MFL}


The question of the mean-field limit of system \eqref{eq:micro1} has been the focus of several works in the recent years \cite{ChibaMedvedev19,KaliuzhnyiMedvedev18,Kuehn20,JabinPoyatoSoler21,PT23}. 
The main requirement is that as $N$ tends to infinity, $(\w_{ij})_{i,j\in\setn}$ has a limit, which can be either a graphon $w\in L^\infty(I\times I)$, or a  more general object. 
Then, the microscopic system can be shown to converge (in a certain sense to be specified) towards a measure 
$\mu$, solution to a partial differential equation of the form
\begin{equation}\label{eq:meanfield}
\begin{cases}
\dsp \partial_t \mu_t^\xi(x) + \nabla_x\cdot \left( \left( \int_{I} \int_{\Omega} w(\xi,\zeta) \phi(x,y) d\mu_t^\zeta(y) d\zeta \right) \mu_t^\xi(x) \right) = 0,\\
\mu_{t=0} = \mu_0,
\end{cases} 
\end{equation}
in which the variable $x\in \Omega$, as in the microscopic model \eqref{eq:micro1}, represents the position (or phase, or opinion), and the newly introduced variable $\xi\in I$ is a continuous representation of the particles' labels, or identities.

Notice in particular the asymmetric roles of of the two variables $x$ and $\xi$. 
The limit equation is a transport equation, in which the probability measure $\mu$ is transported only in the direction of the variable $x$, whereas the variable $\xi$ plays the role of a structure variable. As announced in the introduction, we refer to this limit process as the \textbf{non-exchangeable mean-field limit} of System~\ref{eq:micro1}.

The interpretation of $\mu$ can be twofold. Let $\Delta \xi\subset I$ and $\Delta x\subset \R^d$. 
\begin{itemize}
\item In the probabilistic setting, $\mu_t(\Delta \xi \times \Delta x)= \int_{\Delta \xi\times \Delta x}d\mu_t^\xi(x) d\xi$ represent the probability of finding an agent with label in $\Delta \xi$ and position in $\Delta x$ at time $t$.
\item In the deterministic setting, $\mu_t(\Delta \xi \times \Delta x)$ represents the mass of agents with labels in $\Delta \xi$ and positions in $\Delta x$ at time $t$.
\end{itemize}

Note that proving existence and uniqueness of the solution to \eqref{eq:meanfield} is not always trivial. 
However, since it is not the main focus of this review, we do not go into details, and refer the reader to the cited articles for the proof in each of the frameworks that we mention below. 

The first results for non-exchangeable mean-field limit theory appeared in the framework of \textit{dense} graphs (Section \ref{sec:MFLdense}). These results were then extended to graphs with intermediate or sparse densities \ref{sec:MFLsparse}.

\subsubsection{Dense graphs} \label{sec:MFLdense}

For $N\in \N$, let $\langle G_N, V(G_N), E(G_N)\rangle$ denote a graph of $N$ nodes, whose set of vertices and set of edges are respectively given by $V(G_N)$ and $E(G_N)$.
A sequence of graphs $(G_N)_{N\in\N}$ is called ``dense'' \cite{Lovasz12,LovaszSzegedy06} if the number of edges each graph contains is proportional to the maximal number of edges it can contain, i.e. if
\[
|E(G_N)| = O(|V(G_N)|^2)= O(N^2).
\]
The limit of a convergent dense graph sequence is given by a so-called \textit{graphon} \cite{Lovasz12,LovaszSzegedy06}. Although definitions of graphons vary depending on the context in which they are introduced, in this review we will use the following notion: 
\begin{definition}\label{def:graphon}
A \textit{graphon} is a function $w\in L^\infty(I^2;\R)$.
\end{definition}

\paragraph{Lipschitz graphon.} In the seminal paper \cite{ChibaMedvedev19}, Chiba and Medvedev derive the mean-field limit of the microscopic system \eqref{eq:micro1} with the following regularity assumption: 
\begin{hyp}[Graphon regularity]\label{hyp:wLip}
The graphon $w\in \Lip(I^2;\R)$ is a Lipschitz function.
\end{hyp}
The proof is provided in the specific case of the Kuramoto model, in which the variables $(x_i)_{i\in\setn}$ represent oscillator phases, and evolve on the torus $\Omega=\mathbb{T}$. However, it can easily be extended to the case of particles in $\R^d$.
For each $N\in \N$, we introduce a discretization $(\xi^N_1,\cdots,\xi^N_N)\in I^N$ of the unit interval, with the requirement that for all continuous function $f\in C(I)$, the Riemann sum of $f$ evaluated at $(\xi_i^N)_{i\in\setn}$ converge to the integral of $f$ on $I$, i.e.
\[
\lim_{N\rightarrow\infty} \frac{1}{N} \sum_{i=1}^N f(\xi_i^N) = \int_I f(\xi) d\xi.
\]
Given such a discretizations, each discrete weights $\w_{ij}$ is defined by evaluating the continuous graphon $w$ at the corresponding discretization points: \begin{equation}\label{eq:wijpointwise}
    \w_{ij} := w(\xi^N_i,\xi_j^N).
\end{equation}
Then, the empirical measure associated with a solution $(\x_i)_{i\in\setn}$ to the microscopic system \eqref{eq:micro1} with such interaction weights is defined for all $t\in [0,T]$ by 
\begin{equation}\label{eq:empirmu}
\mu^N_t(\xi,x) := \frac{1}{N} \sum_{i=1}^N \delta_{\x_i(t)}(x) \delta_{\xi_i^N}(\xi).
\end{equation}
Importantly, probability measures in $\P(I\times\Omega)$ are compared via the \textit{Bounded Lipschitz distance $\dbl$ on the product space $I\times\Omega$} , defined for all $\mu_0,\nu_0\in\P(I\times\Omega)$ by 
\[
 \dbl(\mu_0,\nu_0) := \sup\left\{ \int_{I\times\Omega} f(\xi,x) d(\mu_0(\xi,x)-\nu_0(\xi,x))\, | \, f\in \Lip(I\times\Omega), \,  \|f\|_{\Lip}\leq 1, \,  \|f\|_{L^\infty}\leq 1\right\}.
\] 
One can show that given Hyp. \ref{hyp:wLip}
for any $T\in\R_+$, there exists a unique weak solution $\mu\in C([0,T],\P(I\times\Omega))$ to the mean-field equation \eqref{eq:meanfield}. The proof relies on the theory of Neunzert developed for the Vlasov equation \cite{Neunzert78, Neunzert84}, consisting in writing $\mu_t$ as the push-forward of $\mu_0$ via the flow of the vector field $\tilde V[\mu_t](\xi,x):=\left(V[\mu_t](\xi,x), \, 0 \right)$, where
\begin{equation}\label{eq:V}
V[\mu_t](\xi,x):=\int_I \int_{\R^d} w(\xi,\zeta) \phi(x,y) d\mu_t^\zeta(y) d\zeta,
\end{equation}
and in showing the existence and uniqueness of the solution to the associated fixed-point equation.
Moreover, one can show continuity of the solution to \eqref{eq:meanfield} with respect to the initial data in the Bounded Lipschitz distance. This implies the following convergence result:
\[
 \dbl(\mu^N_0,\mu_0)\xrightarrow[N\rightarrow\infty]{} 0 \; \Rightarrow \; \dbl(\mu^N_t,\mu_t)\xrightarrow[N\rightarrow\infty]{} 0, 
\]
where $\mu_t$ is the solution to 
\begin{equation*}
\begin{cases}
\dsp \partial_t \mu_t^\xi(x) + \nabla_x\cdot \left( \left( \int_{I} \int_{\Omega} w(\xi,\zeta) \phi(x,y) d\mu_t^\zeta(y) d\zeta \right) \mu_t^\xi(x) \right) = 0,\\
\mu_{t=0} = \mu_0.
\end{cases} 
\end{equation*}

\paragraph{Non-Lipschitz graphon.} In \cite{KaliuzhnyiMedvedev18}, Kaliuzhnyi-Verbovetsky and Medvedev extend the previous result to non-Lipschitz graphons.
The main idea for this improvement is to exploit the fundamental asymmmetry in the roles played by the variables $x$ and $\xi$.
More specifically, if $\mu_0$ is absolutely continuous with respect to the variable $\xi$, one can show that the measure $\mu_t$ remains absolutely continuous with respect to $\xi$ for all time (as shown in \cite{ChibaMedvedev19}). This allows to rewrite 
$
d\mu_0(\xi,x) = d\mu_0^\xi(x)d\xi,
$
where for all $\xi\in [0,1]$, $\mu_0^\xi\in \P(\Omega)$.
In this way, we consider the set of $\P(\Omega)$-valued functions $\bar{\mathcal{M}}:=\{\mu:I\rightarrow \P(\Omega)\}$ equiped with the $L^1$-Bounded-Lipschitz distance 
\[
\bar{d}(\mu_0,\nu_0) = \int_I \dbl(\mu_0^\xi, \nu_0^\xi) d\xi,
\]
which is a complete metric space.
The Lipschitz assumption on the graphon $w$ can then be relaxed to the much weaker condition
\begin{hyp}[Graphon regularity] The graphon $w\in L^\infty(I^2)$ satisfies
\[\lim_{\delta\rightarrow 0} \int_I |w(\xi+\delta,\zeta)-w(\xi,\zeta)| d\zeta = 0.\]
\end{hyp}
With this assumption, one can show that the vector field $V$ defined by \eqref{eq:V} is Lipschitz continuous in $x$, continuous in $\xi$, and Lipschitz continuous with respect to $\mu_t$. This implies that the equation of characteristics 
$
\dot x(t) = V[\mu_t](\xi,x(t))
$
is well-posed, and so that it generates a flow $\Psi_{t,0}$ defined for almost all $\xi\in I$ by $\Psi_{t,0}^\xi[w, \mu]u(0) = x(t)$.
Then, it can be shown that the fixed point equation 
$
\mu_t = \Psi_{0,t}[w,\mu]\#\mu_0
$
has a unique solution $\mu\in C([0,T],\bar{\mathcal{M}})$, which in turn implies existence and uniqueness of the solution to the mean-field equation \eqref{eq:meanfield}.
Moreover, the solution to \eqref{eq:meanfield} is continuous with respect to its initial data in the $L^1-\BL$ distance $\bar{d}$, and is continuous with respect to the graphon $w$ in the $L^1$-norm.

The link between the microscopic system \eqref{eq:micro} and the mean-field equation \eqref{eq:meanfield} is done via the empirical measure $\nu_{n,m}\in C([0,T],\bar{M})$, defined for all $N := nm$, for all $\xi\in [\frac{i-1}{n},\frac{i}{n})$, for all $t\in [0,T]$ by
\begin{equation}\label{eq:empirnu}
\nu^\xi_{n,m,t}(x) := \frac{1}{m}\sum_{j=1}^m \delta_{\x_{(i-1)m+j}(t)}(x),
\end{equation}
where $(\x_i)_{i\in\setn}$ is the solution to system \eqref{eq:micro1} in which the discrete weights $(\w_{ij})_{i,j\in\setn}$ are given by an $L^1$-approximation of the graphon $w$:
\begin{equation}\label{eq:wij}
\w_{ij} := N^2 \int_{\frac{i-1}{N}}^{\frac{i}{N}}\int_{\frac{j-1}{N}}^{\frac{j}{N}} w(\xi,\zeta) d\xi d\zeta.
\end{equation}
Then, the empirical measure $\nu_{n,m,t}$ can be shown to converge towards the solution $\mu_t$ to the mean-field equation \eqref{eq:meanfield}, in the $L^1$-Bounded-Lipschitz distance. 

\paragraph{General interaction function.}
In \cite{PT23}, Paul and Tr\'elat provide a general result for particle systems of the form 
\begin{equation}
    \begin{cases}
\displaystyle \frac{d}{dt} \x_i(t) = \frac{1}{N}  \sum_{j=1}^N G(t,\frac{i}{N},\frac{j}{N},\x_i(t),\x_j(t)) \quad \text{ for all } i\in\setn, \; t\in [0,T] \\
\x_i(0) = x_i^{N,0}\quad \text{ for all } i\in\setn,
\end{cases}
\end{equation}
in which the effect of the particles' labels is no longer decoupled from that of the particles' positions.
The function 
\[
\begin{array}{rcl}
     G :\R\times I\times I \times \R^d \times \R^d &\rightarrow & \R^d\\
     (t,\xi,\zeta,x,y) & \mapsto & G(t,\xi,\zeta,x,y)
\end{array}
\]
is assumed to be locally Lipschitz with respect to $(x,y)$ uniformly with respect to $(t,\xi,\zeta)$ on any compact subset of $\R\times I\times I$.
 Convergence can be proven to the following Vlasov equation:
\begin{equation}\label{eq:meanfield_PaulTrelat}
\begin{cases}
\dsp \partial_t \mu_t^\xi(x) + \nabla_x\cdot \left( \left( \int_{I\times\R^d} G(t,\xi,\zeta,x,y) d\mu_t^\zeta(y) d\zeta \right) \mu_t^\xi(x) \right) = 0,\\
\mu_{t=0} = \mu_0.
\end{cases} 
\end{equation}
Existence and uniqueness of the weak solution $\mu\in C([0,T];\P_c(I\times\R^d))$ to \eqref{eq:meanfield_PaulTrelat} is proven for all compactly supported initial data $\mu_0\in \P_c(I\times\R^d)$.
Moreover, the link between the discrete and the continuous system is provided by introducing an empirical measure
\begin{equation}\label{eq:empiricalmeasure_PaulTrelat}
    \mu_t^{\xi,N}(x) := \frac{1}{N} \sum_{i=1}^N \delta_{\frac{i}{N}}(\xi) \delta_{x_i(t)}(x).
\end{equation}
Then, denoting by $W_p$ the $p$-Wasserstein distance on the product space $I\times \R^d$, one has the following result for all $p\geq 1$: 
\[ W_p(\mu_0,\mu^N_0)\xrightarrow[]{N\rightarrow \infty} 0 \quad \Rightarrow \quad W_p(\mu_t,\mu^N_t)\xrightarrow[]{N\rightarrow \infty}0
\]
on all compact time interval $[0,T]$. 

Moreover, if $G$ is assumed to be locally Lipschitz with respect to all four arguments $(\xi, \zeta, x, y)$, a stability estimate for \eqref{eq:meanfield_PaulTrelat} allows to write
\[ \forall t\in [0,T], \quad  W_p(\mu_t,\mu^N_t)\leq C_{\mu,\mu^N}(t) W_p(\mu_0,\mu^N_0),
\]
where the constant $C_{\mu,\mu^N}(t)$ depends on the Lipschitz norm of $G$ with respect to $x$ and $y$ on the supports of $\mu_t$ and $\mu_t^N$.

Note that if the system is posed on a graph, i.e. $G(t,\xi,\zeta,x,y):= w(\xi,\zeta) \phi(x,y)$, this strong Lipschitz condition on $G$ implies that the graphon $w$ is Lipschitz, as in \cite{ChibaMedvedev19}. 

 \begin{rem}
    In its full generality, the result from \cite{PT23} is stated for a general $n$-dimensional set $I$, and the Lebesgue measure $d\xi$ in \eqref{eq:meanfield_PaulTrelat} can be replaced by a more general measure $d\nu(\zeta)$. The final time of existence $T$ is not necessarily finite.
\end{rem}

\paragraph{A priori unknown graphon.}
The previous three results \cite{ChibaMedvedev19,KaliuzhnyiMedvedev18,PT23} are all based on a same general idea: a limit graphon $w$ is given, and used to build a converging graph sequence by discretizing it in an appropriate way (see for instance \eqref{eq:wijpointwise} or \eqref{eq:wij}). The convergence of the microscopic system \eqref{eq:micro1} towards the limit Vlasov equation is obtained as a consequence of this discretization procedure. 
This ``top-down'' process can be argued to be somewhat artificial, since the graph sequence is built using the \textit{a priori} knowledge of the limit graphon.

In \cite{BetCoppiniNardi23}, Bet, Coppini and Nardi propose a different approach, with no \textit{a priori} knowledge of the limit object,
for the random system on a graph $G_N$
\begin{equation}\label{eq:micro_BCN}
    \begin{cases}
 \dsp   dX_t^{i,N} = \frac{1}{N} \sum_{i=1}^N \w_{ij} \phi(X^{i,N}_t,X^{j,N}_t) dt + dB^i_t \\
    X_0^{i,N} = X^i_0,
    \end{cases}
\end{equation}
where $(B^i)_{i\in\setn}$ is a sequence of independent and identically distributed (i.i.d.) Brownian motions on $\Omega:=\T$, the initial conditions $(X^i_0)_{i\in\setn}$ are i.i.d.  sampled from some probability distribution $\bar\mu_0$, and the weights $(\w_{ij})$ are the edge weights of the graph $G_N$.

The main tool of this approach is the cut-norm $\|\cdot\|_\square$ (see Equation \eqref{eq:cut-norm}), well-known to the Graph Theory community \cite{Lovasz12,LovaszSzegedy06}.
To prove convergence of the microscopic system with noise \eqref{eq:micro_BCN}, one needs to consider the cut-distance between two graphons, taking into consideration all possible relabelings, defined by
\[
\delta_\square(w,\tilde{w}) = \min_{\varphi\in S_I} \|w-\tilde{w}^\varphi\|_\square,
\]
where $S_I$ denotes the space of invertible measure-preserving maps $\varphi$ from $I$ to $I$, and where $\tilde{w}^\varphi:(\xi,\zeta)\mapsto \tilde{w}(\varphi(\xi), \varphi(\zeta))$.
Importantly, in this setting, the cut-distance between two graphons $w$ and $\tilde{w}$ can be zero for two different graphons. The important point is that the graphons be equal up to relabeling.

In this framework, the microscopic system is linked to its mean-field limit by the traditional empirical measure $\mu_t^N\in \P(\Omega)$ defined as in \eqref{eq:empirical_measure_exchangeable_intro} by
\begin{equation}
    \mu^N_t := \frac{1}{N} \sum_{i=1}^N \delta_{\x_i(t)}.
\end{equation} 
One also needs to define the non-linear process
\begin{equation}\label{eq:nonlinearprocess}
\begin{cases}
 \dsp   X_t = X_0 + \int_0^t \int_I \int_{\Omega} w(U,\zeta) \phi(X_s,y) d\mu_s^\zeta(y) d\zeta ds+ B_t\\
    \mu_t^y = \mathcal{L} ( X_t|U=\zeta) \quad \text{for } \zeta\in I, t\in [0,T],
\end{cases}
\end{equation}
where $\mathcal{L}(X_0)=\bar\mu_0$, and $B_t$ is a Brownian motion.
 
The main result can then be stated as follows. Assume that $\phi\in C^{1+\varepsilon}(\Omega)$ for some $\varepsilon>0$. Consider a sequence of graphs $G_N$, whose associated graphons $w_{G_N}$ defined by 
\[
w_{G_N} :(\xi,\zeta)\mapsto  \sum_{i=1}^N \sum_{j=1}^N \w_{ij} \mathbf{1}_{[\frac{i-1}{N},\frac{i}{N})}(\xi)  \mathbf{1}_{[\frac{j-1}{N},\frac{j}{N})}(\zeta) 
\]
converges in cut-norm to an \textit{a priori} unknown limit graphon in probability, in the sense that 
\[
\lim_{N\rightarrow+\infty} \mathbb{E} [ \delta_\square(w_{G_N}, w)] = 0.
\]
Note that the limit graphon $w$ is not unique, since any relabeling of $w$ would also be a limit of $w_{G_N}$ in the cut-distance. For this reason, $w$ is refered to as an \textit{unlabeled graphon} in \cite{BetCoppiniNardi23}.
Then, the empirical measures $\mu^N$ converge to a limit measure $\bar\mu\in \P(C([0,T],\Omega))$, which is the weak solution to the non-linear Fokker-Plank equation
\begin{equation}
    \begin{cases}
\dsp \partial_t \bar\mu_t(x) + \nabla_x\cdot \left( \int_I \left( \int_{I} \int_{\Omega} w(\xi,\zeta) \phi(x,y) d\mu_t^\zeta(y) d\zeta \right) \mu_t^\xi(x) d\xi \right) = \frac{1}{2} \Delta_x \bar\mu_t(x),\\
\bar\mu_{t=0} = \bar\mu_0,
\end{cases} 
\end{equation}
in which the measure $\mu_t^\xi$ is defined by the non-linear process \eqref{eq:nonlinearprocess}.

We refer the reader to \cite{BetCoppiniNardi23,BhamidiBudhijaratWu19,DelattreGiacominLucon16} for more details on such probabilistic approaches.

\subsubsection{Sparse graphs}\label{sec:MFLsparse}

As seen in the Introduction, the graph sequence corresponding to Example \eqref{eq:Kuramoto_intro_2} admits a graphon limit when $k$ increases proportionally to $N$ as $N$ tends to infinity. What can be said when $k$ is not proportional to $N$?

Recent results in Graph Theory have provided multiple ways of defining the limit of a graph sequence which is not dense. These new limit objects include $L^p$-graphons, graphops, graphings, and measures \cite{BackhauszSzegedy22,BorgsChayesCohnZhao19,KunszentiLovaszSzegedy19}. We refer to \cite{JabinPoyatoSoler21} for an illustration of the relationships between the various limit objects.

\paragraph{Graphop.}
The result of \cite{KaliuzhnyiMedvedev18} is further extended by Gkogkas and Kuehn in \cite{GkogkasKuehn22} to a class of sparse graphs defined using so-called ``graphops'' (graph operators).
Graphops were introduced in~\cite{BackhauszSzegedy22} by Backhausz and Szegedy in the aim of providing a general framework unifying dense and sparse graph theories.  
The main idea consists of moving away from the object of graphons (i.e. functions $w\in L^\infty(I^2)$) and of considering instead the action of graphs as operators from $L^\infty(I)$ to $L^1(I)$.
\begin{definition}
A graphop is a linear operator $A:L^\infty(I)\rightarrow L^1(I)$ 
which has finite operator norm $\|A\|_{\infty\rightarrow1}$, is positivity-preserving, and is self-adjoint. 
\end{definition}
Importantly, 
there exists a family $(\nu_A^\xi)_{\xi\in I}$ of finite fiber measures so that the action of the graphop $A$ on a function $f\in L^\infty(I)$ is given by: 
\[
Af:\xi \mapsto \int_I f(\zeta) d\nu_A^\xi(\zeta).
\]
This object generalizes the concept of symmetric graphons in the following way: for every symmetric graphon $w\in L^\infty(I^2)$, one can define a graphop $A_w:L^\infty(I)\rightarrow L^1(I)$ such that for all $f\in L^\infty(I)$, 
\[
A_wf:\xi \mapsto \int_I w(\xi,\zeta) f(\zeta) d\zeta.
\]
Consider then a sequence of graphons $(w_K)_{K\in\N}$ such that their associated graphops $(A_{w_K})_{K\in\N}$ converge as $K$ goes to infinity to a limit graphop $A$, in the sense that for almost all $\xi\in I$, $\nu^\xi_{A_K}\rightharpoonup \nu^\xi_{A}$.
As shown in \cite{GkogkasKuehn22}, one can derive the mean-field limit of the solution $x^{N,K}$ of the microscopic system~\eqref{eq:micro1}
in which the weights $w^{N,K}_{ij}$ are given by an $L^1$ approximation of $w_K$ as in \eqref{eq:wij}, given the necessary regularity condition:
\begin{hyp}[Graphop regularity]
For all $\xi,\xi_0\in I$, if $\xi\rightarrow\xi_0$, then $\nu^{\xi}_{A}\rightharpoonup \nu^{\xi_0}_{A}$.
\end{hyp}
The empirical measure $\nu_{n,m,K,t}$ defined as in \eqref{eq:empirnu} from $x^{N,K}$ can then be shown to converge towards $\mu^\xi_t$, defined for all $S\in\mathcal{B}(\R)$ by $\mu^\xi_t(S) := \int_S \rho(t,\xi,x)dx$, where $\rho$ is the solution to the mean-field equation
\[
\displaystyle \partial_t \rho(t,\xi,x) + \nabla_x\cdot \left(\rho(t,\xi,x) \left(\int_{\Omega} \phi(x,y) (A\rho)(t,\xi,y)dy \right) \right) = 0, 
\]
in the $L^1$-Bounded-Lipschitz distance.

\begin{rem}
In \cite{GkogkasKuehn22}, the proof is done on $\Omega=\mathbb{T}$ for the Kuramoto model, but can be easily extended to $\R$.
Moreover, the space of nodes can be extended from $(I,d\xi)$ with $I=[0,1]$ as exposed here, to a more general $(\tilde I,dm(\xi))$, with $\tilde I\subset \R^n$.
\end{rem}

\begin{rem}
As explained above, this approach requires the converging sequence of graphons to be symmetric (i.e. the corresponding graphs to be undirected).
\end{rem}

\paragraph{Digraph measures.} In \cite{KuehnXu22}, Kuehn and Xu further generalize the result of \cite{GkogkasKuehn22} to sparse directed graphs (also called ``digraphs''), using the framework of \textit{digraph measures}. Denoting by $\mathcal{M}_+$ the set of finite Borel positive measures, one can define digraph measures as everywhere-defined in the first variable, bounded measures in the second variable: 
\begin{definition}
Any measure-valued function $\eta\in \mathcal{B}(I,\mathcal{M}_+(I))$ is a \textit{digraph measure}.
\end{definition}
Convergence is shown in the \textit{uniform bounded Lipschitz metric} $d_\infty$, defined for all $\mu_1$, $\mu_2\in \mathcal{B}(I,\mathcal{M}_+(\R^d))$ by 
\[
d_\infty(\mu_1,\mu_2) := \sup_{\xi\in I} \dbl(\mu_1^\xi,\mu_2^\xi).
\]
As in the works \cite{KaliuzhnyiMedvedev18} and \cite{GkogkasKuehn22}, a continuity assumption with respect to the first variable is required of the limit object (in this case, a digraph measure) to which the sequence of graphs converge: 
\begin{hyp}
$\eta\in \mathcal{C}(I,\mathcal{M}_+(I))$
\end{hyp}

Using this framework, a \textit{graphon} can be viewed as the digraph limit of a sequence of dense graphs, and a \textit{graphop} is a symmetric digraph measure.

The results presented in \cite{KuehnXu22} apply to general compact label sets $I$ (not necessarily $[0,1]$), with a reference measure that can differ from the usual Lebesgue one, and in particular that can be discrete or singular. 
However, in order to be consistent with what has been presented above, we present the mean-field limit equation that is obtained in this simplified framework, which can be written as
\[
\displaystyle \partial_t \rho(t,\xi,x) + \nabla_x\cdot \left(\rho(t,\xi,x) \left(\int_I \int_{\Rd} \phi(x,y) \rho(t,\zeta,y)  dy d\eta^\xi(\zeta) \right) \right) = 0.
\]

\paragraph{Extended graphon.}
The approach proposed by Jabin, Poyato and Soler in \cite{JabinPoyatoSoler21} differs from the approaches in \cite{KaliuzhnyiMedvedev18}, \cite{GkogkasKuehn22} and \cite{KuehnXu22}, in that it requires no continuity of the limit object (in this new framework, called \textit{extended graphon}), which is defined by
\begin{definition}
An extended graphon is a measure $w\in L^\infty_\xi\mathcal{M}_\zeta\cap  L^\infty_\zeta\mathcal{M}_\xi$. 
\end{definition}
Importantly, similarly to the last paragraph of Section \ref{sec:MFLdense}, this limit procedure requires no \textit{a priori} knowledge of the limit of the discrete coupling weights $(\w_{ij})_{i,j\in\setn}$.
Consequently, instead of making assumptions on the limit extended graphon, assumptions are made on the discrete weights:
\begin{hyp}
\begin{itemize} The discrete weights $(\w_{ij})_{i,j\in\setn}$ satisfy:
\item $\dsp \max_{1\leq i\leq N} \sum_{j=1}^N |\w_{ij}| = O(1)$ and $\dsp \max_{1\leq j\leq N} \sum_{i=1}^N |\w_{ij}| = O(1)$ 
\item $\dsp \max_{1\leq i,j\leq N} |\w_{ij}| = o(1)$. 
\end{itemize}
\end{hyp}
Given such discrete weights, the microscopic system \eqref{eq:micro1} is shown to converge to a limit function $\mu\in L^\infty([0,T]\times I, W^{1,1}\cap W^{1,\infty}(\Omega))$ for any $T>0$, solution to 
\[
\dsp \partial_t \mu_t^\xi(x) + \nabla_x\cdot \left( \left( \int_I w(\xi,d\zeta) \int_{\Omega}  \phi(x,y) \mu_t^\zeta(dy) \right) \mu_t^\xi(x) \right) = 0,
\]
where $w\in L^\infty_\xi\mathcal{M}_\zeta\cap  L^\infty_\zeta\mathcal{M}_\xi$ is an \textit{extended graphon}. More precisely, convergence is obtained in the following sense: up to the extraction of a subsequence, 
\[
\lim_{N\rightarrow\infty} \sup_{0\leq t\leq T} \mathbb{E} W_1\left( \int_I \mu_t^\xi(\cdot)d\xi,\mu_t^N \right) = 0,
\]
where $W_1$ is the $1$-Wasserstein distance on $\Omega$ and the empirical measures $\mu^N_t$ are defined from the solution $\x$ to the discrete system \eqref{eq:micro1} by 
\begin{equation}\label{eq:empirmux}
\mu_t^N(x):= \frac{1}{N} \sum_{i=1}^N\delta_{\x_i(t)}(x).
\end{equation}
Notice that symmetry is required neither of the discrete weights $(\w_{ij})_{i,j\in\setn}$ nor of the extended graphon $w$, unlike in the \textit{graphop} approach \cite{GkogkasKuehn22}. However, contrarily to the digraph measure framework, a symmetric role is given to $\xi$ and $\zeta$ in the definition of the extended graphon $w\in L^\infty_\xi\mathcal{M}_\zeta\cap  L^\infty_\zeta\mathcal{M}_\xi$.

\subsection{Continuum limit} \label{sec:Graphs_GL}

The mean-field limit presented in Section \ref{sec:Graphs_MFL} provides \textit{weak} convergence of the microscopic system \eqref{eq:micro1} towards the solution to a transport equation \eqref{eq:MFL}. Another approach, denoted \textit{continuum limit} or \textit{graph limit}, provides a \textit{pointwise} convergence of the solution to \eqref{eq:micro1} towards the solution $x\in C([0,T]; L^\infty(\R^d))$ to an integro-differential Euler-type equation: 
\begin{equation}\label{eq:GL1}
\begin{cases}
\dsp \partial_t x(t,\xi) = \int_I w(\xi,\zeta) \phi(x(t,\xi),x(t,\zeta)) d\zeta \\
x(0,\cdot) = x_0,
\end{cases}
\end{equation}
also denoted by \textit{nonlinear heat equation} in \cite{Medvedev14}.

In the case of the mean-field limit, the link between the solution $\x$ to the microscopic system \eqref{eq:micro1} and the solution $\mu$ to the limit equation \eqref{eq:MFL} is done via a so-called \textit{empirical measure} (given for example in equations \eqref{eq:empirmu}, \eqref{eq:empirnu}, \eqref{eq:empirmux}). In the continuum limit framework, the link between the solution $x^N$ to \eqref{eq:micro1} and the solution $x$ to \eqref{eq:GL1} is provided by constructing a piecewise-constant function $x_N\in C([0,T]; L^\infty(\R^d))$ from $\x$: 
\begin{equation}\label{eq:discrete-to-continuous}
\forall \xi\in  I, \quad x_N(t,\xi) := \sum_{i=1}^N \x_i(t) \mathbf{1}_{[\frac{i-1}{N},\frac{i}{N})}(\xi).
\end{equation}
The crucial point consists of noticing that $x_N$ is solution to \eqref{eq:GL1} if and only if $\x$ is solution to the microscopic equation \eqref{eq:micro1} with the graphon $w^N$ given by 
\begin{equation}\label{eq:wn}
\forall (i,j)\in \setn^2, \quad \w_{ij} := \int_{\frac{i-1}{N}}^{\frac{i}{N}} \int_{\frac{j-1}{N}}^{\frac{j}{N}} w(\xi,\zeta) d\zeta d\xi.
\end{equation}

Another key difference with the mean-field limit approach is that all results currently published are in the context of \textit{dense graphs}, i.e. graphs whose limits are graphons $w\in L^\infty(I^2)$ (see Definition~\ref{def:graphon} in Section~\ref{sec:MFLdense}).
The available results can be divided into results for deterministic graphs and results for random graphs. 

\subsubsection{Continuum limit on deterministic graphs}
\label{sec:ContinuousLimitDet}

A seminal paper providing the first proof of convergence of the solution to the microscopic equation~\eqref{eq:micro1} towards its continuum limit \eqref{eq:GL1} was published in 2014 by Medvedev \cite{Medvedev14}. 
The main result can be stated as follows. Let $x\in C([0,T];L^\infty(\R^d))$ denote the solution to the integro-differential equation~\eqref{eq:GL1}
on a given graphon $w\in L^\infty(I^2)$ and with an initial condition $x_0\in L^\infty(I)$.
For each $N\in \N$, consider the solution $x^N$ to the microscopic system \eqref{eq:micro1} on the underlying graph whose adjacency matrix $\w$ is given by \eqref{eq:wn}
and the initial condition by 
\[
\forall i\in \setn, \quad x^{N,0}_{i} := \int_{\frac{i-1}{N}}^{\frac{i}{N}} x_0(\xi) d\xi.
\]
Let $x_N\in C([0,T];L^\infty(\R^d))$ denote the piecewise-constant function built from the vector $\x$ by \eqref{eq:discrete-to-continuous}. Then, $x_N$ converges to $x$ in $L^2-$norm, satisfying
\[
\lim_{N\rightarrow\infty} \sup_{t\in [0,T]} \| x(t,\cdot) - x_N(t,\cdot)\|_{L^2(I)} = 0.
\]
A more detailed explanation of this limit process is presented in Section \ref{sec:GL_adaptive}, in the case of a particle system with time-evolving weights.

Remarkably, this result requires no continuity of the graphon $w$.
However, as shown in \cite{PT23} (Theorem 4), if both the initial data and the graphon are regular enough, one can obtain a quantitative convergence result. 
More precisely, let $\alpha\in (0,1]$ such that $x_0\in C^{0,\alpha}(I;\R^d)$. Let $G:[0,T]\times I^2\times(\R^d)^2$ be locally Lipschitz with respect to its last two arguments uniformly with respect to the first three, and suppose in addition that $G$ is locally $\alpha$-H\"older with respect to  its last four arguments.
Then, the solution $x$ to the integro-differential equation
\begin{equation*}
\begin{cases}
\dsp \partial_t x(t,\xi) = \int_I G(t,\xi,\zeta,x(t,\xi),x(t,\zeta)) d\zeta \\
x(0,\cdot) = x_0,
\end{cases}
\end{equation*}
satisfies $x(t,\cdot)\in C^{0,\alpha}(I;\R^d)$ for all $t\in [0,T]$.
Moreover, let $x^N$ denote the solution to the particle system
\begin{equation*}
\begin{cases}
\displaystyle \frac{d}{dt} \x_i(t) = \frac{1}{N}  \sum_{i=1}^N G(t,\frac{i}{N}, \frac{j}{N},\x_i(t),\x_j(t)) \quad \text{ for all } i\in\setn \\
\dsp \x_i(0) = x_0(\frac{i}{N}).
\end{cases}
\end{equation*}
Then, for every $N\in\N$, it holds
\[
\max_{i\in\setn} |x(t,\frac{i}{N})-\x_i(t)| \leq \frac{1}{N^{\alpha}}(1+\mathrm{Hol}_\alpha(x_0))e^{2tL_x^N(t)},
\]
where $\mathrm{Hol}_\alpha(x_0)$ denotes the H\"older constant of $x_0$ and $L_x^N(t)$ 
depends on the H\"older constant of $G(\tau,\cdot,\cdot,\cdot,\cdot)$ and on the Lipschitz constant of $G(\tau,\xi,\zeta,\cdot,\cdot)$ for all $\tau\in [0,t]$.

Moreover, letting $x_N$ be the piecewise-constant function defined by \eqref{eq:discrete-to-continuous} from the solution $x^N$ to the microscopic system, it holds
\[
\|x(t,\cdot)-x_N(t,\cdot)\|_{L^\infty(I)} \leq \frac{2}{N^{\alpha}}(1+\mathrm{Hol}_\alpha(x_0))e^{2tL_x^N(t)}.
\]

\begin{rem}
In order to insist on the role of the regularity, we stated the result from Theorem 4 in \cite{PT23} in a simplified form. The full result, as stated in its most general form, only requires $I$ to be a general compact smooth $n$-dimensional manifold, and the final time $T\in \R_+\cup\{+\infty\}$ corresponds to the uniform maximal time until which the solution to the microscopic system is well-defined. We encourage the reader to consult \cite{PT23} for a complete statement.
\end{rem}


\subsubsection{Continuum limit on random graphs}\label{sec:ContinuousLimitRand}

\paragraph{Random unweighted graphs.}

In \cite{Lovasz12,LovaszSzegedy06}, Lovasz and Szegedy introduce a method to construct random graphs from a graphon $w$. Such random graphs, named $w$-random graphs, are unweighted, meaning that each edge is either present or absent, and consequently, 
the corresponding adjacency matrix $(w_{ij})_{i,j\in\setn}$ contains values in $\{0,1\}$. 
 Given a graphon $w\in L^\infty(I^2;[0,1])$, a $w$-random graph can be defined in two different ways, as proposed in \cite{MedvedevR}: either from a sequence of i.i.d. random variables, or from a sequence of deterministic evenly-spaced variables.
 Although these differences may seem subtle, they are indeed fundamental and the type of convergence one may hope to obtain depends on the degree of randomness (``random-random'' or ``random-deterministic'') introduced in the graph construction.
 \begin{definition} \label{def:wrandomgraph}
  Let $w\in L^\infty(I^2;[0,1])$. A $w$-random graph can be constructed either from a sequence of random variables (r-r) or from a sequence of deterministic variables (r-d):
 \begin{enumerate}
 \item[(r-r)] Let $Z=(Z_i)_{i\in\N}$ be a sequence of i.i.d. random variables, uniformly distributed in $I=[0,1]$. Let $N\in \N$. A random (unweighted) graph $G_N$ generated by the random sequence $Z$ is constructed by inserting each edge~$(i,j)$ with probability $w(Z_i,Z_j)$:
 \[\mathbb{P}[(i,j)\in E(G_N)] = w(Z_i,Z_j).\]
 The corresponding adjacency matrix is given by $(w_{ij})_{i,j\in\setn}$ with 
 \[
 w_{ij} = 
  \begin{cases}
  1 \text{ if } (i,j)\in E(G_N)\\
  0 \text{ otherwise.}
 \end{cases}
 \]
 \item[(r-d)] Let $N\in \N$. Let $\tZ=(\tZ^N_i)_{i\in\setn}$ be a determistic sequence satisfying $\tZ^N_i\in [\frac{i-1}{N} , \frac{i}{N})$ for all $i\in\setn$. A random (unweighted) graph $\tilde{G}_N$ generated by the deterministic sequence $\tZ$ is constructed by inserting each edge~$(i,j)$ with probability $w(\tZ^N_i,\tZ^N_j)$:
  \[\mathbb{P}[(i,j)\in E(\tilde{G}_N)] = w(\tZ^N_i,\tZ^N_j).\]
 The corresponding adjacency matrix is given by $(w_{ij})_{i,j\in\setn}$ with 
 \[
 w_{ij} = 
  \begin{cases}
  1 \text{ if } (i,j)\in E(\tilde{G}_N)\\
  0 \text{ otherwise.}
 \end{cases}
 \]
 \end{enumerate}
\end{definition}

Two convergence results can then be given, one in each of the settings (r-r) and (r-d).

Firstly, in \cite{MedvedevR}, Medvedev proves that given a symmetric graphon $w\in L^\infty(I^2,[0,1])$, a Lipschitz function $\phi:\R\rightarrow\R$ and $x_0\in L^\infty(I)$, if the solution $x\in C([0,T];L^\infty(I))$ to \eqref{eq:GL1} satisfies the inequality 
\[
\sup_{t\in [0,T]}\int_I\left(\int_I w(\xi,\zeta) \phi(x(t,\xi),x(t,\zeta))^2 d\zeta - \left(\int_I w(x,y) \phi(x(t,\xi),x(t,\zeta)) d\zeta\right)^2\right)d\xi \geq C_1
\]
for some constant $C_1>0$, then the solution $x^N$ to the microscopic system \eqref{eq:micro1} posed on the $w$-random graph generated by a random sequence (r-r) constructed as in Definition \ref{def:wrandomgraph} converges towards $x$. The convergence is obtained in the following sense: for some $C>0$,
\[
\lim_{N\rightarrow+\infty} \PP\left[ \sup_{t\in [0,T]} \|\x(t)-\mathbf{P}_{Z^N}x(t,\cdot)\|_{2,N} \leq C\right] = 1,
\]
where for all $a\in \R^N$,
$\|a\|_{2,N}:=\frac{1}{N}\sum_{i=1}^N a_i$, and 
 $\mathbf{P}_{Z^N}x(t,\cdot):=(x(t,Z_1),\cdots,x(t,Z_N))$ denotes the evaluation of the function $x(t,\cdot)\in L^\infty(I)$ at the discretization points $(Z_1,\cdots,Z_N)$. 

On the other hand, the convergence of the microscopic system on a $w$-random graph generated by a deterministic sequence (r-d) requires more regularity from the graphon $w$. More specifically, in \cite{MedvedevR}, Medvedev  proves that given a symmetric graphon $w\in L^\infty(I^2;[0,1])$ almost everywhere continuous, a Lipschitz function $\phi:\R\rightarrow\R$ and $x_0\in L^\infty(I)$, if the solution $x\in C([0,T];L^\infty(I))$ to \eqref{eq:GL1} satisfies the inequality 
\[
\sup_{t\in [0,T]} \int_I^2 w(\xi,\zeta) (1-w(\xi,\zeta))\phi(x(t,\xi)-x(t,\zeta)) d\xi d\zeta > 0
\] 
for some $T>0$, then defining $x_N$ from the solution $x^N$ to the microscopic system \eqref{eq:micro1} on the $w$-random graph (r-d) as by the relation \eqref{eq:discrete-to-continuous}, it holds
\[
\|x_N - x\|_{C([0,T];L^2(I))} \xrightarrow[N\rightarrow+\infty]{\mathbb{P}} 0,
\]
where the convergence is in probability. 

The proofs of both results reliy on several applications of the Central Limit Theorem, and, in the case (r-d), on the introduction of an intermediate deterministic system, known to converge due to results in \cite{Medvedev14} (See Section \ref{sec:ContinuousLimitDet}).

\paragraph{Random weighted graphs.}
A generalization of this convergence result to general directed weighted random graphs is provided in \cite{AyiPouradierDuteil23}.
In that aim, the concept of $w$-random graphs is generalized to that of $q$-weighted random graphs as follows:

\begin{definition} \label{def:qweightedrandomgraph}
  Let $q:I^2\rightarrow \P(\R_+)$. A $q$-weighted random graph can be constructed either from a sequence of random variables (r-r) or from a sequence of deterministic variables (r-d).
 \begin{enumerate}
 \item[(r-r)] Let $Z=(Z_i)_{i\in\N}$ be a sequence of i.i.d. random variables, uniformly distributed in $I=[0,1]$. Let $N\in \N$. A $q$-weighted random graph generated by the random sequence $Z$ is constructed by randomly attributing to each edge~$(i,j)$ a weight $w_{ij}\in\R_+$ with law $q(Z_i,Z_j,\cdot)$.
 \item[(r-d)] Let $N\in \N$. Let $\tZ=(\tZ^N_i)_{i\in\setn}$ be a determistic sequence satisfying $\tZ^N_i\in [\frac{i-1}{N} , \frac{i}{N})$ for all $i\in\setn$. A $q$-weighted random graph generated by the deterministic sequence $\tZ$ is constructed by randomly attributing to each edge~$(i,j)$ a weight $w_{ij}\in\R_+$ with law $q(\tZ_i,\tZ_j,\cdot)$.
 \end{enumerate}
\end{definition}

Again, the convergence results depend heavily on the degree of randomness ((r-r) or (r-d)) of the $q$-weighted random graph.

In the case of a $q$-weighted random graph generated by a random sequence, given a Lipschitz function $\phi:\R\rightarrow\R$ and an initial condition $x_0\in L^\infty(I;\R)$, the solution $\x$ to the discrete system posed on the $q$-weighted random graph converges to the solution $x$ to the continuous equation \eqref{eq:GL1}, where the limit graphon is the first moment of the weighted random graph law $q$: 
\begin{equation}\label{eq:wq}
\forall (\xi,\zeta)\in I^2, \quad w(\xi,\zeta) = \int_{\R_+} q(\xi,\zeta;du).
\end{equation}
The convergence is obtained quantitatively in the following sense: 
\[
\PP\left[ \sup_{t\in [0,T]} \|\x(t)-\mathbf{P}_{Z^N}x(t,\cdot)\|_{2,N} \geq \frac{C_1(T)}{\sqrt{N}}\right] \leq \frac{\tilde{C}_1}{N}
\]
for some constants $C_1(T)$ and $\tilde{C}_1$.

In the case of a $q$-weighted random graph generated by a deterministic sequence, more regularity is required both of the initial data and of the first moment of $q$, the weighted random graph law. More specifically, if $x_0\in C^{0,\frac{1}{2}}(I)$ and $(\xi,\zeta)\mapsto w q(\xi,\zeta;dw)$ is $\frac{1}{2}-$Hölder on $I^2$, then the solution $\x$ to the microscopic system \eqref{eq:micro1} posed on the $q$-weighted random graph converges to the solution $x$ to the continuous equation \eqref{eq:GL1}, where the limit graphon is the first moment of the weighted random graph law $q$, as given by equation \eqref{eq:wq}. Denoting by $x_N$ the projection of the vector $\x$ onto $C([0,T];L^\infty(I))$ given by \eqref{eq:discrete-to-continuous}, it holds
\[
\PP\left[ \|x_N - x\|_{C([0,T];L^2(I))} \geq \frac{C_2(T)}{\sqrt{N}}\right] \leq \frac{\tilde{C}_2}{N}
\]
for some constants $C_2(T)$ and $\tilde{C}_2$.

Both proofs of convergence rely on the Bienaym\'e-Chebyshev inequality, and in the case (r-d), on the convergence of an intermediate deterministic system given by Theorem 4 in \cite{PT23} (see Section \ref{sec:ContinuousLimitDet}).

\subsection{Links between Continuum limit and Mean-Field limits}
\label{sec:Link}

Let $N\in\N$. Denoting by $\x_i(t)\in\Rd$ the position of particle $i$ at time $t$, a general particle system can be described by $(\x_i)_{i\in\setn}\in C([0,T],\Rd)^N$, whose evolution is given by: 
\begin{equation}\label{eq:micro}
\begin{cases}
\displaystyle \frac{d}{dt} \x_i = \frac{1}{N}  \sum_{i=1}^N \phi(\frac{i}{N}, \frac{j}{N}, \x_i, \x_j) \quad \text{ for all } i\in\setn \\
\x_i(0) = x_i^{N,0}.
\end{cases}
\end{equation}

As seen in Section \ref{sec:Graphs_GL}, its limit as $N$ goes to infinity can be written as the solution $x\in C([0,T],L^2(I;\Rd))$ to the following integro-differential equation, provided that $x_i^{N,0}$ converges to $x^0$ in a suitable sense: 
\begin{equation}\label{eq:GL}
\begin{cases}
\displaystyle \partial_t x(t,\xi) = \int_I \phi(\xi,\zeta, x(t,\xi), x(t,\zeta)) d\zeta\\
x(0,\cdot) = x^0.
\end{cases}
\end{equation}

We refer to $x$ as the \textbf{continuum limit} of $\x$. Within this framework, the infinitely numerous agents are assumed to be labeled by the variable $\xi$, which spans the set $I=[0,1]$. Thus, $x(t,\xi)$ denotes the position of agent with label $\xi$ at time $t$.

Secondly, as seen in Section \ref{sec:Graphs_MFL}, taking the \textbf{non-exchangeable mean-field} limit of the microscopic system \eqref{eq:micro} yields the following  equation on the probability density $\mu\in C([0,T],\P(I\times\Rd))$:
\begin{equation}\label{eq:MFL}
\begin{cases}
\displaystyle \partial_t \mu_t^\xi(x) + \nabla_x\cdot \left( \left(\int_{I\times\Rd} \phi(\xi,\zeta,x,y) d\mu_t^\zeta(y) d \zeta\right) \; \mu_t^\xi(x) \right) = 0 \\
\mu_{t=0} = \mu_0.
\end{cases}
\end{equation}

Thirdly, if the agents are indistinguishable, i.e. if $\phi(\xi,\zeta,x(t,\xi),x(t,\zeta))= \phi(x(t,\xi),x(t,\zeta))$, we can derive the classical \textbf{exchangeable mean-field limit} giving the evolution of the probability measure $\mu\in C([0,T],\P(\Rd))$:
\begin{equation}\label{eq:MFLx}
\begin{cases}
\displaystyle \partial_t \mu_t(x) + \nabla_x\cdot \left( \left(\int_{\Rd} \phi(x,y) d\mu_t(y)\right) \; \mu_t(x) \right) = 0 \\
\mu_{t=0} = \mu_0.
\end{cases}
\end{equation}
Here, $\mu(\Omega)$ denotes the mass of agents in the space region $\Omega$.

All three equations \eqref{eq:GL}, \eqref{eq:MFL} and \eqref{eq:MFLx} are obtained as the limit of \eqref{eq:micro} as $N$ goes to infinity, as shown in Sections \ref{sec:Graphs_MFL} and \ref{sec:Graphs_GL}. Here, we explain briefly and formally how these descriptions are related to one another  (see Figure \ref{schema_static} for a visual summary). A complete and detailed overview of how all models are related can be found in \cite{PT23}.

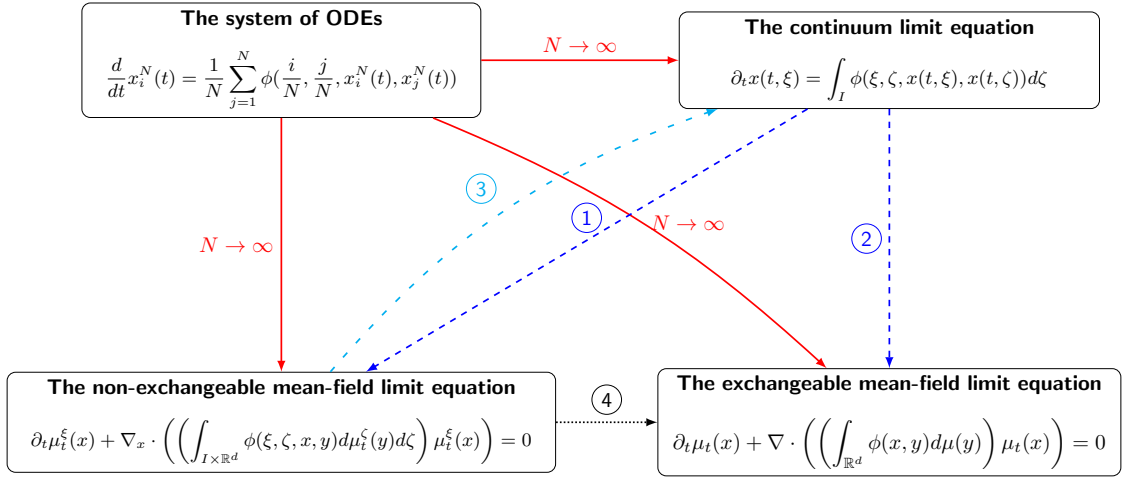
\begin{figure}
\begin{center}
\scalebox{0.8}{
\begin{tikzpicture}[>=latex,font=\sffamily]
\tikzstyle{block1} = [draw, rectangle, text width=18em, text centered, minimum height=4em, rounded corners]
\tikzstyle{block2} = [draw, rectangle, text width=19em, text centered, minimum height=4em, rounded corners]
\tikzstyle{block3} = [draw, rectangle, text width=25em, text centered, minimum height=4em, rounded corners]
\tikzstyle{block4} = [draw, rectangle, text width=25em, text centered, minimum height=4em, rounded corners]
\tikzstyle{block4} = [draw, rectangle, text width=25em, text centered, minimum height=4em, rounded corners]
\tikzstyle{block5} = [draw, rectangle, text width=21em, text centered, minimum height=4em, rounded corners]
\tikzstyle{arrow} = [->, thick]
\tikzstyle{arrow1} = [->, thick,blue,dashed]
\tikzstyle{arrow2} = [->, thick,red]
\tikzstyle{arrow3} = [->, thick,densely dotted]

\node [block1] (A1) at (0,0) {\textbf{The system of ODEs} \small \begin{equation*}
\displaystyle \frac{d}{dt} x_i^{N}(t) = \frac{1}{N}  \sum_{j=1}^N \phi(\frac{i}{N},\frac{j}{N},x_i^{N}(t),x_j^{N}(t))
\end{equation*}};
\node [block2] (A2) at (10,0) {\textbf{The continuum limit equation} \small  \begin{equation*}
\displaystyle \partial_t x(t,\xi) = \int_I \phi(\xi,\zeta,x(t,\xi), x(t,\se)) d\se
\end{equation*} };
\node [block3] (B1) at (0,-6) {\textbf{The non-exchangeable mean-field limit equation} \small \begin{equation*}
\displaystyle \partial_t \mu_t^\xi(x) + \nabla_x\cdot \left( \left(\int_{I\times\R^d}   \phi(\xi,\zeta,x,y) d\mu_t^\zeta(y)d\zeta\right)  \mu_t^\xi(x) \right) = 0 
\end{equation*}};
\node [block5] (B3) at (10,-6) {\textbf{The exchangeable mean-field limit equation} 
\begin{equation*}
\displaystyle\partial_t \mu_t(x) + \nabla\cdot \left( \left(\int_{\R^d} \phi(x,y) d\mu(y)\right)\mu_t(x)\right)= 0
\end{equation*}};

\draw [arrow2] (A1) -- node[above] {$N \to \infty$ }  (A2);
\draw [arrow2] (A1) -- node[left] {$N \to \infty$ } (B1);
\draw [arrow, bend left=10,red] (A1) to node[right]{$N \to \infty$} (B3);
\draw [arrow1] (A2) -- node[left] {\circled{2}} (B3);
\draw [arrow1] (A2) -- node[above] {\circled{1}} (B1);
\draw [arrow, bend left=15,cyan,loosely dashed] (B1) to node[above left] {\circled{3}} (A2);
\draw [arrow3] (B1) -- node[above] {\circled{4}} (B3);
\end{tikzpicture}
}
\end{center}
\caption{Links between the different equations\label{schema_static}. The red arrows show the large-population limits described in Sections \ref{sec:Graphs_MFL} and \ref{sec:Graphs_GL}. The dashed arrows 1, 2 and 3 are explained in Sections \ref{sec:Cont2NEMFL}, \ref{non_exchangeable_to_graph_limit} and \ref{sec:subordination}. Arrow 4 corresponds to Remark \ref{rem:arrow4}.}
\end{figure}

\subsubsection{From continuum limit to non-exchangeable mean-field limit}\label{sec:Cont2NEMFL}

Let $x(t,\xi)$ denote the solution to \eqref{eq:GL}, and let $\tmu_t$ denote an empirical measure defined by 
\[
\tmu_t(\xi,x) = \int_I \delta_{x(t,\zeta)}(x) \delta_{\zeta}(\xi) d\zeta.
\] 
For all test functions $f\in C^\infty(I\times\Rd)$, it holds
\[
\begin{split}
& \frac{d}{dt}\int_{I\times\Rd} f(\xi,x) d\tmu_t^\xi(\x) d \xi 
= \frac{d}{dt}\int_{I} f(\xi,x(t,\xi)) d\xi \\
= \quad & \int_{I} \nabla_x f(\xi,x(t,\xi))\cdot \left( \int_I \phi(\xi, \zeta,x(t,\xi),x(t,\zeta)) d\zeta \right)  d\xi \\
= \quad & \int_{I\times\Rd} \nabla_x f(\xi,x)\cdot \left( \int_{I\times\Rd} \phi(\xi, \zeta,x,y) d\tmu_t^\zeta(y) d \zeta  \right)  d\tmu_t^\xi(x) d \xi, 
\end{split}
\]
which shows that the continuous empirical measure $\tmu$ built from the solution $x$ to \eqref{eq:GL} is itself a solution to \eqref{eq:MFL}. This computation corresponds to the arrow 1 in Figure   \ref{schema_static}.

Note that the set of measures $\mu_t\in \P(I\times\Rd)$ that can be written as a continuous empirical measure $\tmu_t$ are the measures supported on curves $\xi\mapsto x(t,\xi)$.

\subsubsection{From non-exchangeable mean-field limit to continuum limit}
\label{non_exchangeable_to_graph_limit}

As shown in \cite{PT23}, a solution to \eqref{eq:GL} can be recovered from a solution to the non-exchangeable mean-field limit equation \eqref{eq:MFL} by taking its first moment with respect to the space variable. 

In this aim, we suppose that the marginal of $\mu_0$ with respect to its first variable is the Lebesgue measure, i.e. that
\[
\pi_I\#\mu_0 = d\xi.
\] 
As observed in \cite{PT23} (Remark 1), the marginal of $\mu_t$ with respect to its first variable is constant in time, since the transport term acts only on the space variable. It then holds $\pi_I\#\mu_t = d\xi$ for all $t\in [0,T]$.  
Let us then write the disintegration of $\mu_t$ with respect to its marginal $d\xi$ as
\[
\mu_t = \int_I \mu^\xi_t \, d\xi,
\]
and further assume the following
\begin{hyp}
 $\int_{\Rd} d\mu^\xi_t(x) = 1$ for almost all $\xi\in I$.
\end{hyp}
We then define the first moment of $\mu_t$ with respect to the space variable as
\[
\bx(t,\xi) := \int_{\Rd} x \, d\mu^\xi_t(x). 
\]
Decomposing $\bx$ on its coordinates using an orthonormal basis $(e_k)_{k\in\setd}$ of $\Rd$, it then holds for all $k\in\setd$
\[
\begin{split}
 \frac{d}{dt}\bx(t,\xi) \cdot e_k
= & \frac{d}{dt}\int_{\Rd} (x\cdot e_k) \, d\mu^\xi_t(x) 
= \int_{\Rd} \nabla_x(x\cdot e_k) \cdot \left(\int_{I\times\Rd} \phi(\xi,\zeta,x,y) d\mu^{\zeta}_t(y) d \zeta \right) \, d\mu^\xi_t(x) \\
= &  \int_{\Rd}  e_k \cdot \left(\int_{I\times\Rd} \phi(\xi,\zeta,x,y) d\mu^{\zeta}_t(y) d \zeta \right) \, d\mu^\xi_t(x).
\end{split}
\]

At this point, we are led to make a strong assumption on the form of the interaction function $\phi$.
\begin{hyp}
 We suppose that $(\xi,\zeta,x,y)\mapsto\phi(\xi,\zeta,x,y)$ can be written as the product of a function of $\xi,\zeta$ and of a linear function of $x,y$, i.e.
 \[
\phi(\xi,\zeta,x,y) = w(\xi,\zeta) (\lambda_1 x+\lambda_2 y),
\]
with $w:I^2\mapsto\R$ and $(\lambda_1, \lambda_2)\in\R^2$.
\end{hyp}
This form is actually common in models for opinion dynamics with linear interaction of the type Hegselmann-Krause (see \cite{HK}), for which the interation is given by $w(\xi,\zeta) (y-x)$.
We obtain
\[
\begin{split}
 \frac{d}{dt}\bx(t,\xi) 
= &  \int_{\Rd}   \left(\int_{I\times\Rd} w(\xi,\zeta)(\lambda_1 x +\lambda_2 y) d\mu^{\zeta}_t(y) d \zeta \right) \, d\mu^\xi_t(x) \\
= & \int_{I}  w(\xi,\zeta)\left(\lambda_1 \int_{\Rd} x d\mu^\xi_t(x) +\lambda_2 \int_{\Rd} y d\mu^{\zeta}_t(y)\right) d\zeta
=  \int_{I}  w(\xi,\zeta)\left(\lambda_1 \bx(t,\xi) +\lambda_2 \bx(t,\zeta) \right) d\zeta
\end{split}
\]
which is the continuum limit equation \eqref{eq:GL} for this specific choice of $\phi$. This computation corresponds to the arrow 3 in Figure   \ref{schema_static}.

The fact that one can recover a closed equation on the first moment of $\mu_t$ with respect to the space variable is remarkable, and can be intuitively explained as follows. 
In the general case, the evolution of $\bx(t,\xi)$ depends on all remaining particles labeled by all $\zeta\in I$ and located at all $y\in \Rd$.
The linear condition on $\phi$ implies that the combined effect of all the particles with label $\zeta$ is equivalent to that of their first moment $\bx(t,\xi)$.
The problem of obtaining a closed equation in the general (nonlinear) case is still open, and we refer the reader to \cite{PT23} (Section 3.1.2.) for further comments on this issue.

\begin{rem}
The assumption on the marginal of $\mu_0$ with respect to its first variable can be lifted without loss of generality. If $\pi_I\#\mu_0:=\nu$ for a general $\nu\in\P(I)$, one can recover the general graph-limit equation
\[
\displaystyle \partial_t x(t,\xi) = \int_I \phi(\xi,\zeta, x(t,\xi), x(t,\zeta)) d\nu(\zeta).
\]
\end{rem}

\subsubsection{Subordination of the mean-field limit to the continuum limit equation (indistinguishable case)}\label{sec:subordination}

In the case in which the particles are indistinguishable, i.e. $\phi(\xi,\zeta,x,y) = \phi(x,y)$, one can derive the classical mean-field limit equation \eqref{eq:MFLx} as the limit as $N$ tends to infinity of the particle system \eqref{eq:micro}. 
As shown in \cite{BiccariKoZuazua19}, Equation \eqref{eq:MFLx} can also be obtained directly from the continuum limit \eqref{eq:GL} using the ``continuous'' empirical measure 
\[
\bmu_t(x) = \int_I \delta_{x(t,\xi)}(x) d\xi
\] 
where $x(t,\xi)$ is a solution to \eqref{eq:GL}. Indeed, 
for all test functions $f\in C^\infty(\Rd)$, it holds
\[
\begin{split}
 \frac{d}{dt}\int_{\Rd} f(x) d\bmu_t(x) 
= &\frac{d}{dt}\int_I f(x(t,\xi)) d\xi
= \int_{I} \nabla_x f(x(t,\xi))\cdot \left( \int_I \phi(x(t,\xi),x(t,\zeta)) d\zeta \right)  d\xi \\
=  & \int_{\Rd} \nabla_x f(x)\cdot \left( \int_{\Rd} \phi(x,y) d\bmu_t(y) \right)  d\bmu_t(x),
\end{split}
\]
which is the weak formultion of \eqref{eq:MFLx}. This computation corresponds to the arrow 2 in Figure \ref{schema_static}.

Thus, deriving the continuum limit of \eqref{eq:micro} as $N$ goes to infinity and doing the transformation above is an alternative way of obtaining the classical mean-field limit equation \eqref{eq:MFLx}. 

Two important remarks need to be made. 

Firstly, all initial measures $\mu_0\in\P(\Rd)$ can be approximated by empirical measures of the form $\bmu_t(x) = \int_I \delta_{x(t,\xi)}(x) d\xi$, as shown in the following proposition.
\begin{prop}
The set of measures 
\[
\left\{ \int_I \delta_{x_0(\xi)}(x) d\xi\quad | \quad x_0:I\rightarrow\Rd \text{ measurable }\right\}
\]
is dense in $\P(\Rd)$ for the weak topology.\\
Moreover, for if $d=1$, for any $\mu_0\in\P(\R)$, there exists a measurable function $x_0:I\rightarrow \R$ such that $\int_I \delta_{x_0(\xi)}(x) d\xi$.
\end{prop}
\begin{proof}
It is well known that the set of measures that can be written as finite sums of Dirac masses is dense in $\P(\Rd)$.
We claim that for any $\bmu_0^n := \sum_{i=1}^n a_i \delta_{x_i}$ where $a_i\in [0,1]$ for all $i\in\{1,\cdots,n\}$ and $\sum_{i=1}^n a_i = 1$, there exists a function $x_0:I\rightarrow\Rd$ such that 
\[
\bmu_0^n(x) = \int_I \delta_{x_0(\xi)}(x) d\xi.
\]
We prove this claim constructively. Let $x_0$ be defined by the piecewise-constant function
\[\forall i\in\{1,\cdots,n\},\; \forall \xi\in [b_i,b_{i+1}),\; x_0(\xi) = x_i   
\]
where $b_i:= \sum_{k=1}^{i-1} a_k$.
Then 
\[
\int_I \delta_{x_0(\xi)}(x) d\xi = \sum_{i=0}^{n} \int_{b_i}^{b_i+a_{i}} \delta_{x_i}(x) d\xi = \sum_{i=1}^{n} a_{i} \delta_{x_i}.
\]
\end{proof}

Secondly, the choice of $x_0$ is in general not unique. Any measure-preserving rearrangement of $x_0$ gives the same mesure $\bmu_0$.
This implies that in passing from the solution $x(t,\xi)$ to the continuum limit equation to the solution $\mu_t(x)$ to the mean-field equation, there is an irreversible information loss. In this indistinguishable case, the agents can be relabeled without impacting the dynamics.

\begin{rem}\label{rem:arrow4}
    In the indistinguishable case, the exchangeable mean-field equation \eqref{eq:MFLx} can be recovered from the non-exchangeable mean-field equation \eqref{eq:MFL} by integrating the solution $\mu^\xi_t$ in $\xi$.
\end{rem}


\section{Large-population limits of particle systems on adaptive dynamical networks}
As mentioned in the introduction, the setting of adaptive dynamical networks is of great interest since it allows to build more realistic models. Indeed, in many real-life situations, the graph involved is not static and its evolution can depend on the  states of the system itself. 
The natural question is to see if and how the large-population limit considered above, i.e. the non-exchangeable and exchangeable mean-field limits and the continuum limit can be extended to this setting. Actually, it is a challenging problem in its generality, and up to recently, two specific frameworks have successfully addressed it.

The first one is the setting of Kuramoto-type models on $N$ oscillators \cite{gkogkas2023}. It can be written in the following form 
\begin{equation}\label{eq:kuramoto}
\begin{cases}
\displaystyle \frac{d}{dt} x_i^N = \omega_i^N( x_i^N,t) +\frac{1}{N}  \sum_{j=1}^N \w_{ij}  \phi\left( x_i^N, x^N_j\right) \quad \text{ for all } i\in\setn \\
\displaystyle \frac{d}{dt} w_{ij}^N =  - \varepsilon \left(w_{ij}^N + H( x_i^N, x^N_j) \right)
\end{cases}
\end{equation}
where $x_i^N = x_i^N(t) \in \mathbb{T} = \R / (2 \pi \mathbb{Z})$ represents the phase of the $i$-th oscillator for $ i \in \{1, \dots, N\}$,  $\omega_i^N: \mathbb{T} \times \mathbb{R} \to \mathbb{R}$ is the vector field describing the intrinsic frequency, $\phi : \mathbb{T}^2 \to \R$  is a coupling function, $w^N(t)=(w^N_{ij}(t))_{1 \leq i,j \leq N}$ is the time-evolving weight matrix of the network of oscillators, which takes into account the local information of two
interacting oscillators via the function $H: \mathbb{T}^2 \to \R$  and $\varepsilon >0$ is a parameter which controls the time scale between the dynamics on the network of the phases and the dynamics of the network weights. 

The second one can be seen as a variant of the Hegselmann-Krause dynamics \cite{HK} where, additionally to the opinions, we also are interested in the time-evolving weights of agents which represent their charisma \cite{AyiPouradierDuteil21}. This model can also be viewed as a system of ODEs on a evolving-weighted (on-symmetric) graph (see Figure \ref{fig:charisma_graph}). It can be written in the following way 
\begin{equation}\label{eq:syst-gen}
\begin{cases}
\displaystyle \frac{d}{dt} x_i^{N}(t) = \frac{1}{N}  \sum_{j=1}^N m_j^N(t)\, \phi(x_j^{N}(t)-x_i^{N}(t))\\
\displaystyle  \frac{d}{dt} m_i^{N}(t) = \psi_i^{(N)}(x^{N}(t),m^{N}(t))
\end{cases}
\end{equation}
where $x^N=(x_1^{N}, \dots, x_N^{N}):[0,T]\rightarrow(\R^d)^N$ represent the {opinions} of $N$ agents, and  $m^N=(m_1^{N}, \dots, m_N^{N}):[0,T]\rightarrow \R^N$ represent their individual {weights of influence}. Each opinion's time-evolution is affected by the opinion of each neighboring agent via the interaction function $\phi$,
proportionally to the neighboring agent's weight of influence. Moreover, the agents' weights are assumed to evolve in time and their dynamics may depend on the opinions and weights of all the other agents, via functions $\psi_i^{(N)}:(\R^d)^N\times\R^N\rightarrow\R$.
\begin{figure}[h!]
\begin{center}
\includegraphics[scale=0.25]{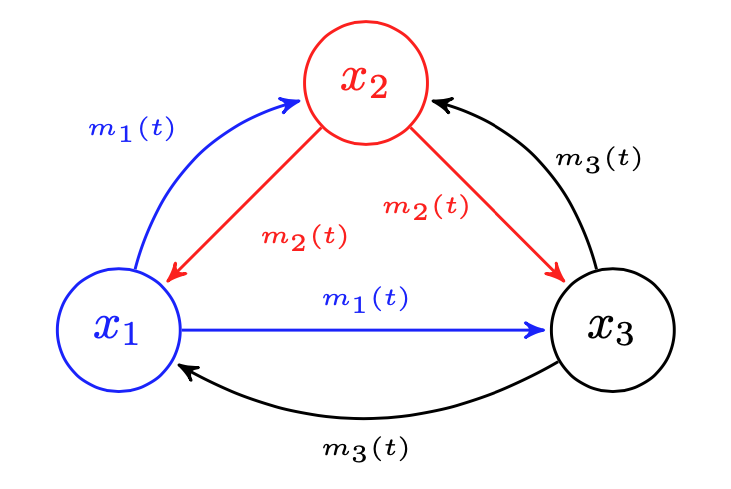}
\end{center}
\caption{Evolving-weighted graph associated to \eqref{eq:syst-gen}  in the case of three agents.}\label{fig:charisma_graph}
\end{figure}

More recently, in \cite{throm2023continuum}, inspired by the approach developed for the particular model of evolving-in-time weights in \cite{AyiPouradierDuteil21}, some results have been obtained for a more general Kuramoto-type model of the form 
\begin{equation}\label{eq:kuramoto-gen}
\begin{cases}
\displaystyle \frac{d}{dt} x_i^N = \omega_i^N( x^N,t) +\frac{1}{N}  \sum_{j=1}^N \w_{ij}  \phi\left( x_i^N, x^N_j\right) \quad \text{ for all } i\in\setn \\
\displaystyle \frac{d}{dt} w_{ij}^N =   \psi_{ij}^{(N)}(x^{N}(t),w^{N}(t))
\end{cases}
\end{equation}

As in the static case, when considering large population limit, several approaches are possible: mean-field and continuum limits. We start with the mean-field one.

\subsection{Mean-field limits}
\label{mfl_adaptive}
\subsubsection{The non-exchangeable mean-field limit}
\label{subsec:non_exchangeable_mfl}
Let us start by recalling that the question of the derivation of a non-exchangeable mean-field limit is very difficult if we want to deal with general adaptive networks. So far, there exists only partial results on very specific cases.

\paragraph{The Kuramoto-type model.} In \cite{gkogkas2023mean}, Gkogkas, Kuehn and Xu are interested in a particular form of \eqref{eq:kuramoto} that we write below
\begin{equation}\label{eq:kuramotobis}
\begin{cases}
\displaystyle \frac{d}{dt} x_i^N = \omega_i^N +\frac{1}{N}  \sum_{j=1}^N \w_{ij}  \varphi \left( x_j^N-x^N_j\right) \quad \text{ for all } i\in\setn \\
\displaystyle \frac{d}{dt} w_{ij}^N =  - \varepsilon \left(w_{ij}^N + h( x_j^N-x^N_i) \right).
\end{cases}
\end{equation}
Note that the authors explain that their result can actually extend to the more general form \eqref{eq:kuramoto} easily. Let us go back to the framework of digraph measures defined in subsection \ref{sec:MFLsparse} that we recall are measure-valued function $\eta\in \mathcal{B}(I,\mathcal{M}_+(I))$. We introduce the following generalized adaptive Kuramoto network 
\begin{equation}
\left\{\begin{array}{l}
\displaystyle \partial_t x(t,\xi) = \omega(t,\xi) + \int_I \int_\T \varphi(y - x(t,\xi)) d \nu_t^\zeta(y) d \eta_t^\xi(\zeta), \\
\displaystyle \partial_t  \eta_t^\xi(\cdot)= - \varepsilon  \eta_t^\xi(\cdot) - \varepsilon \int_\T h(y-x(t,\xi)) d\nu_t^\cdot (y) \lambda(\cdot)
\end{array}\right.
\label{eq:generalized_kuramoto}
\end{equation} 
which is interpreted in a weaker sense in integral form. The equation on $\eta^\xi$ is to be understood in the weak sense, meaning that for any bounded continuous test function $f \in \mathcal{C}_b(I)$, $\eta_t^\xi$ satisfies
\begin{multline}
\int_I f(\zeta)  d\eta_t^\xi(\zeta) = \int_I f(\zeta)  d\eta_0^\xi(\zeta) - \varepsilon \int_{t_0}^t \left(\int_I f(\zeta)  d\eta_\tau^\xi(\zeta)  \right) d\tau \\
- \varepsilon  \int_{t_0}^t  \int_I \left(f(\zeta) \int_\T h(y-x(\tau,\xi)) d \nu_\tau^\zeta(y) d \zeta\right)d\tau.
\end{multline}
The key idea of the article is that, thanks to a variation of constants formula for $\eta$, one can prove that being a solution to \eqref{eq:generalized_kuramoto} in its integral form is equivalent to being a solution to 
\begin{equation}
\left\{\begin{array}{l}
\displaystyle x(t,\xi) = x_0(\xi)+ \int_{t_o}^t \left( \omega(s,\xi) + e^{-\varepsilon(s-t_0)}  \int_I \int_\T \varphi(y - x(s,\xi)) d \nu_s^\zeta(y) d \eta_0^\xi(\zeta)\right., \\
\displaystyle~~~~~~~~~~~- \varepsilon \int_{t_0}^s e^{-\varepsilon(t-\tau)}  \int_I \left( \int_\T \varphi(y - x(s,\xi)) d \nu_s^\zeta(y) \right.\\
\displaystyle~~~~~~~~~~~~~~~~~~~~~~~~~~~~~~~~~~~~~~~~\left. \left. \cdot \int_\T h(y-x(\tau,\xi))d\nu_\tau^\zeta(y)\right) d\zeta \right)\\
\displaystyle \eta_t^\xi(\cdot)= e^{-\varepsilon(t-t_0)}  \eta_0^\xi(\cdot)- \left(\varepsilon \int_{t_0}^t e^{-\varepsilon(t-s)} \int_\T h(y-x(s,\xi)) d\nu_s^\cdot (y)ds\right) \lambda(\cdot)
\end{array}\right.
\label{eq:generalized_kuramoto_int}
\end{equation} 
Under this form, the dynamics of the oscillators are decoupled from the dynamics of the digraph measures. One can then reduce the hybrid system to a one-dimensional integral equation indexed by the vertex variable coupled on the prescribed initial graph measures as well as prescribed time-dependent measure valued functions. The mean-field limit equation is a generalized Vlasov equation:
\begin{multline}
\label{eq:Vlasov_generalized_kuramoto}
 \partial_t \rho(t,\xi,x) + \partial_x\left(\rho(t,\xi,x) \left(\omega(t,\xi) +   e^{-\varepsilon(t-t_0)}   \int_I \int_{\T} \varphi(y-x) \rho(t,\zeta,y)  dy d\eta_0^\xi(\zeta) \right. \right. \\
\left. \left. +    \int_I \int_{\T} \varphi(y-x) \rho(t,\zeta,y)  dy\left(\varepsilon \int_{t_0}^t e^{-\varepsilon(t-s)} \int_\T h(z-x) \rho(s,\zeta,z)dz ds\right)d\zeta \right) \right) 
=0
\end{multline}
Their result consist in establishing approximations of the solution to the  mean-field limit  \eqref{eq:Vlasov_generalized_kuramoto} by empirical distributions generated by a sequence of ODEs like \eqref{eq:kuramotobis} thanks to stability estimates and discretization of a given initial digraph measure and of the initial condition of the generalized Vlasov equation. 

\paragraph{The evolving-in-time weight model.} So far, there is no result of non-exchangeable mean-field limit for the general system \eqref{eq:syst-gen}. However, for weight dynamics of the form
\begin{equation}
\psi_i(x^N,m^N) = \int_{\{1,\dots,N\} \times \R^d \times \R} \tilde{S}(i,x_i^N,m_i^N,k,y,n)d\tilde{\mu}_t^N(k,y,n) 
\end{equation}
where $\tilde{S}$ is locally Lipschitz in $(x_i^N,m_i^N,y,n)$ uniformly in $(i,k)$ and $$\tilde{\mu}_t^N(k,x,m) = \frac{1}{N}  \sum_{i=1}^N  \delta(x-x^{N}_i(t))  \delta(m-m^{N}_i(t))  \delta(k-i),$$ the convergence to the following equation on the probability density $\mu \in \mathcal{C}([0,T], \mathcal{P}(I \times \R^d \times \R))$
\begin{multline}\label{eq:non_exchangeable_mfl}
\displaystyle \partial_t \mu_t^\xi(x,m) + \nabla_x\cdot \left( \left(\int_{I\times\R^d \times \R}  n \phi(x,y) d\mu_t^\zeta(y,n)d\zeta\right)  \mu_t^\xi(x,m) \right) \\
+ \partial_m \left( \left( \int_{I\times\R^d \times \R} \tilde{S}(\xi,x,m,\zeta,y,n)  d\mu_t^\zeta(y,n)d\zeta \right)  \mu_t^\xi(x,m) \right) = 0 
\end{multline}
can be established applying the result proved in \cite{PT23}. Indeed, in that case, by setting $X_i=(x_i,m_i)$, we can rewrite the system as 
\begin{equation}
\begin{cases}
\displaystyle \frac{d}{dt} X_i^N = \frac{1}{N}  \sum_{i=1}^N \Phi(\frac{i}{N}, \frac{j}{N}, X_i^N, X_j^N) \quad \text{ for all } i\in\setn \\
X_i(0) = X_i^{N,0}
\end{cases}
\end{equation}
where $\Phi$ satisfy the necessary regularity assumptions.

\subsubsection{The exchangeable mean-field limit}
In this subsection, we will only focus on the evolving-in-time weight model. It was first introduced in \cite{McQuadePiccoliPouradierDuteil19}, with a focus on its long-time behavior for specific choices of weight dynamics. The first considerations regarding the large-population limit have been addressed in \cite{PiccoliPouradierDuteil21} for a more general model, and rigorously proven in \cite{AyiPouradierDuteil21,PouradierDuteil21}.  

In this review, we denote by \textit{exchangeable mean-field limit} the classical mean-field limit, which initially appeared in the context of gas dynamics (see \cite{braun1977} for instance), and which only applies to systems that preserve \textit{indistinguishibility} (or \textit{exchangeability}, see Def. \ref{def:exchangeability}). 
Recall that for classical particle systems, the empirical measure $\nu^N_t\in \P(\R^d)$ is defined by
\begin{equation*}
\nu^N_t(x) := \frac{1}{N} \sum_{i=1}^N \delta(x-x^{N}_i(t)).
\end{equation*}
Importantly, $\nu^N_t$ stays unchanged for any relabeling of the particles, which justifies that it can be used as a link between a particle system and a Vlasov-type equation only if the particle system preserves exchangeability.

In the case of System  \eqref{eq:syst-gen}, due to the specific role played by the weights, it has been chosen to work with a modified empirical measure $\hat{\mu}^N_t\in \P(\R^d)$, defined by
\begin{equation}\label{eq:empmes}
\hat{\mu}^N_t(x) := \frac{1}{N} \sum_{i=1}^N m^{N}_i(t) \delta(x-x^{N}_i(t)).
\end{equation}
Notice that in addition to being blind to any relabeling of the agents, this new empirical measure stays unchanged if the weights of agents at the same location $x$ are redistributed between them. 
Thus, the notion of indistinguishibility in that case needs to include the preservation of ``grouping of the agents'': the dynamics must not be influenced by redistribution of the weights of grouped agents.

For this reason, 
 the mean-field limit has been proved for the particular class of weight dynamics which preserves this new notion of indistinguishability, given by
\begin{equation}
\label{eq:psi_i_gen}
\psi_i^{(N)}(x^N,m^N) = m_i^N(t)\frac{1}{N^k} \sum_{j_1=1}^N \cdots \sum_{j_k=1}^N m_{j_1}^N(t)\cdots m_{j_k}^N(t) S(x_i^N(t), x_{j_1}^N(t), \cdots x_{j_k}^N(t)),
\end{equation} 
where $k\in\N$ and $S:(\R^d)^{k+1}\rightarrow \R$ is globally bounded and Lipschitz. 
Another assumption is that the total weight of the system is conserved, i.e. $\sum_i \psi_i^{(N)} = 0$, in order for $\hat{\mu}^N_t$ to be a probability measure for all $t\geq 0$.

\begin{rem} 
This last assumption is formulated as a skew-symmetry property on $S$ in \cite{AyiPouradierDuteil21,PouradierDuteil21}, 
which indeed leads to the conservation of the total weight.
\end{rem}


Thus, for this special class of weights, it can be proved that the empirical measure converges to a measure $\mu_t\in \P(\R^d)$  solution  to the following transport equation with source: 
\begin{equation}
\label{eq:mfl_weight}
\partial_t \mu_t(x) + \nabla\cdot ( V[\mu_t](x) \mu_t(x)) = h[\mu_t](x)
\end{equation}
with the non-local vector-field given by
\[
\forall \mu\in\PR,\quad \forall x\in\R^d,  \quad V[\mu](x) = \int_{\R^d} \phi(y-x) d\mu(y),
\]
the non-local source term given by
\[
\forall \mu\in \PR, \quad \forall x\in\R^d, \quad h[\mu](x) = \left(\int_{(\R^d)^k} S(x,y_1,\cdots,y_k) d\mu(y_1)\cdots d\mu(y_k)\right) \mu(x).
\]

\paragraph{Regular interaction function.}
For  $\phi\in\Lip(\R^d; \R)$, this result is actually proven by two different approaches. In \cite{PouradierDuteil21}, the proof consists of noticing that the empirical measure is already a solution to \eqref{eq:mfl_weight} and in establishing continuity with respect to the initial data, implying that 
\[
 \dbl(\hat{\mu}^N_0,\mu_0)\xrightarrow[N\rightarrow\infty]{} 0 \; \Rightarrow \; \dbl(\hat{\mu}^N_t,\mu_t)\xrightarrow[N\rightarrow\infty]{} 0, 
\]
where $\mu_t$ is the solution to \eqref{eq:mfl_weight} with initial data $\mu_0$. The other proof, in \cite{AyiPouradierDuteil21}, is actually based on exploiting the continuum limit (see section \ref{Links2}).

\paragraph{Singular interaction function.} In \cite{BenPorat2023}, the authors are concerned with analyzing the mean-field limit for \eqref{eq:syst-gen} in the particular case of the attractive 1D Coulomb interaction where $\phi(x) = V'(x)$ with $V(x)=|x|$, i.e. $\phi(x)=\text{sgn}(x)$ with the convention $\text{sgn}(0)=0$. 
More precisely, they study the following system 
\begin{equation}\label{eq:singular_interaction}
\begin{cases}
\displaystyle \frac{d}{dt} x_i^{N}= \frac{1}{N}  \sum_{j=1}^N m_j^N\, \text{sgn}(x_j^{N}-x_i^{N})\\
\displaystyle  \frac{d}{dt} m_i^{N}(t) = \frac{1}{N} m_i^N \sum_{j=1}^N m_j^N S(x_j^N-x_i^N)
\end{cases}
\end{equation}
where $S \in \mathcal{C}_0^\infty(\R)$ is assumed to be odd. 

The mean-field equation for this ODE system is then the following
\begin{equation*}\label{eq:mfl_singular_interaction}
\partial_t \mu_t(x) - \partial_x \left( \left( \int_{- \infty}^x \mu_t(dy) - \int_x^{+ \infty} \mu_t(dy)\right) \mu_t(x)\right) = \mu_t(x) (S \star \mu_t)(x).
\end{equation*} 
We start by noticing that the odd character of $S$ implies that $\mu_t$ stays a probability density for all times. Thus, the above equation can be rewritten as 
\begin{equation}\label{eq:mfl_singular_interaction2}
\partial_t \mu_t(x) - \partial_x \left( \left( 2 \int_{- \infty}^x \mu_t(dy) - 1\right) \mu_t(x)\right) = \mu_t(x) (S \star \mu_t)(x).
\end{equation} 
The main idea developed in \cite{BenPorat2023} is to shift the focus and instead of studying \eqref{eq:mfl_singular_interaction2}, to analyze the equation for the primitive of $\mu_t$. Let us set $$F(t,x) := - \frac{1}{2} + \int_{- \infty}^x \mu_t(dy),$$ then integrating \eqref{eq:mfl_singular_interaction2} yields the following non-local Burgers-type equation for F
\begin{equation}\label{eq:burgers}
\partial_t F + \partial_x (A(F))= \mathbf{S}[F](t,x)
\end{equation}
where $A(F)= -F^2$ and $$\mathbf{S}[F](t,x) := F(t,x) (\varphi \star F)(t,x) - \int_{- \infty}^x  F(t,z) (\partial_z \varphi \star F)(t,z) dz$$ with $\varphi := \partial_x S$.

One of the advantages of this new formulation is that the flux term in the new Burgers equation is local. The strategy is to consider a primitive of the empirical measure associated to \eqref{eq:singular_interaction} $$F_N(t,x) := - \frac{1}{2} + \frac{1}{N} \sum_{j=1}^N m_j(t) H(x-x_j(t))$$ where $H$ is the Heaviside function. One can then establish that $F_N$ is an entropy solution to 
\begin{equation}\label{eq:Burgers_approx}
\partial_t F_N + \partial_x (A_N(t,F_N))= \mathbf{S}[F_N](t,x)
\end{equation}
where $A_N$ is a time-dependent approximation of the time-independent flux $A$. This proof relies on classical arguments like the use of Rankine-Hugoniot and Oleink conditions adapted to this setting (see \cite[Appendix A]{BenPorat2023} for more details). Thee mean-field limit is proven by extracting a converging subsequence and establishing some stability estimates. More precisely, provided that $F_N^0$ converges to $F^0$ in $L^1(\R)$, $F_N$ can be proven to converge to $F$ in $\mathcal{C}([0,T], L^1(\R))$ where $F_N$ and $F$ are respectively the solutions to \eqref{eq:Burgers_approx} and \eqref{eq:burgers} associated with the respective initial data $F_N^0$ and $F^0$.

\subsection{Continuum limit}\label{sec:GL_adaptive}

As mentioned in the static case, the study of continuum limits is of great interest. For instance, in \cite{medvedev_sw}, the analysis of the steady states on small world graphs was carried out by using those continuum limits. Thus, because of the relevance of adaptive dynamical networks from a modeling point of view, it seems natural to extend those limits to that setting. 

\paragraph{Regular weight dynamics.} The first result of this type was established in \cite{AyiPouradierDuteil21}, where the following limit equation
\begin{equation}
\left\{\begin{array}{l}
\displaystyle \partial_t x(t,\xi) = \int_I m(t,\zeta)\phi(x(t,\se) - x(t,\xi)) d\se \\
\partial_t m(t,\xi) = \psi(\xi,x(t,\cdot),m(t,\cdot)),
\end{array}\right.
\label{eq:GraphLimit-gen}
\end{equation} 
was obtained as the continuum limit of the particle system with time-varying weights \eqref{eq:syst-gen}. Here,
 $\xi$ represents the continuous index variable taking values in $I:=[0,1]$, as introduced in the previous sections. 
The proof requires some regularity assumptions and some bounds.
\begin{hyp}\label{hyp:psi}
The function $\psi:I\times L^\infty(I;\R^d)\times L^\infty(I;\R)$ is assumed to satisfy the following Lipschitz properties: there exists $L_\psi>0$ such that for all $(x_1,x_2,m_1,m_2)\in L^2(I)^4$,
\begin{equation}\label{eq:psilip2}
\begin{cases}
 \|\psi(\cdot,x_1,m_1)-\psi(\cdot,x_2,m_1)\|_{L^2(I)} \, \leq \,L_\psi \|x_1-x_2\|_{L^2(I)}\\
 \|\psi(\cdot,x_1,m_1)-\psi(\cdot,x_1,m_2)\|_{L^2(I)} \, \leq \, L_\psi \|m_1-m_2\|_{L^2(I)}.
\end{cases}
\end{equation}
Assume also that there exists $C_\psi>0$ such that for all $(x,m)\in L^\infty(I,\R^d\times\R)$, for all $\xi\in I$, 
\begin{equation}\label{eq:psisublin}
|\psi(\xi, x, m)| \leq C_\psi( 1+\|m\|_{L^\infty(I)}).
\end{equation}
\end{hyp}
\noindent while  $\phi\in\Lip(\R^d; \R)$.\\
The sublinear growth assumption \eqref{eq:psisublin} may seem restrictive, it is actually necessary in order to prevent the blow-up in finite-time of the weight function $m$. Moreover, it provides a framework which is coherent with the one developed in~\cite{Medvedev14} on graphs with $L^\infty$ weights. Indeed, one can view system \eqref{eq:GraphLimit-gen} as the evolution of the opinions $x$ on a time dependent weighted non-symmetric graph with weights $w(t,\xi,\zeta)=m(t,\zeta)$.

The result follows the same path as in the static case. Let us define
\begin{equation}
\label{eq:xN}
\begin{cases}
\displaystyle x_N(t,\xi) := \sum_{i=1}^N x_i^{N}(t) \mathbf{1}_{[\frac{i-1}{N}, \frac{i}{N})}(\xi)\\
\displaystyle m_N(t,\xi) := \sum_{i=1}^N m_i^{N}(t) \mathbf{1}_{[\frac{i-1}{N}, \frac{i}{N})}(\xi).
\end{cases}
\end{equation}
 Consider the weight dynamics \eqref{eq:syst-gen} with $\psi_i$ defined using the functional $\psi$ appearing in \eqref{eq:GraphLimit-gen} as follows: 
 \begin{equation}\label{eq:psi}
\forall i\in\elts,\qquad \psi_i^{(N)}(x^{N}(t),m^{N}(t)) = N \int_{\frac{i-1}{N}}^{\frac{i}{N}} \psi(\xi,x_N(t,\xi),m_N(t,\xi)) d\xi.
\end{equation}
 The main observation relies on the fact that there is an equivalence between being a solution to \eqref{eq:syst-gen} and the fact that the associated $x_N$ and $m_N$ defined in \eqref{eq:xN} satisfy the following system of integro-differential equations 
\begin{equation}\label{eq:syst-contN}
\begin{cases}
\displaystyle \partial_t x_N(t,\xi) = \int_I m_N(t,\se)\, \phi(\txn(t,\se)-\txn(t,\xi)) \,d\se\\
\displaystyle \partial_t \tmn(t,\xi) = N \int_{\frac{1}{N}\floor{\xi N}}^{\frac{1}{N}(\floor{\xi N}+1)} \psi(\se, \txn(t,\cdot),\tmn(t,\cdot)) \,d\se.
\end{cases}
\end{equation}
Through some fixed-point argument, one can prove that there exists a unique solution $(x^N,m^N)$ to \eqref{eq:syst-gen} and $(x,m)$ to \eqref{eq:GraphLimit-gen} that respectively belong to $\mathcal{C}([0,T];\R^{dN} \times \R^N)$ and  $\mathcal{C}([0,T];L^\infty(I;\R^{dN} \times \R^N))$. Then, as in \cite{Medvedev14}, the proof relies on some $L^2$-estimates of $x_N-x$ and $m_N - m$. More precisely, using the regularity assumptions, some bounds, Cauchy-Schwarz and Gronwall inequalities, it holds 
\begin{multline}
\sup_{t\in [0,T]} \left(\|(x_N-x)(t)\|_{L^2(I)}+\|(m_N - m)(t)\|_{L^2(I)} \right) \\ \leq \left(\|(x_N-x)(0)\|_{L^2(I)}+\|(m_N - m)(0)\|_{L^2(I)}+\int_0^T  \|g_N(\tau)\|_{L^2(I)}d\tau \right) e^{KT}
\end{multline}
for some $K>0$ with 
$$
g_N:(t,\xi) \mapsto N \int_{\frac{1}{N}\floor{\xi N}}^{\frac{1}{N}(\floor{\xi N}+1)}  \psi(\se,x(t),m(t))  \,d\se  - \psi(\xi ,x(t),m(t)).
$$
Thus, defining the initial conditions for the microscopic dynamics as 
\begin{equation}
\label{eq:ICx}
\begin{cases}
\displaystyle x^{0,N} := \left( N \int_{\frac{i-1}{N}}^{\frac{i}{N}} x_0(s) ds \right)_{i \in \{1, \dots, N \}} \in (\R^d)^N \\
\displaystyle  m^{0,N}  := \left( N \int_{\frac{i-1}{N}}^{\frac{i}{N}} m_0(s) ds \right)_{i \in \{1, \dots, N \}} \in (\R)^N,
\end{cases}
\end{equation}
the result is obtained using Lebesgue's differentiaton theorem and the dominated convergence theorem for the term containing $g_N$. Indeed, it implies the convergence in respectively $\mathcal{C}([0,T];L^2(I;\R^{dN}))$ and $\mathcal{C}([0,T];L^2(I;\R^{N}))$ of $x_N$ and $m_N$ to $x$ and $m$, solutions to \eqref{eq:GraphLimit-gen} with initial data $x(0,\cdot)=x_0$ and $m(0,\cdot)=m_0$.

\paragraph{Regular edge dynamics.} The second result of this type was established in \cite{gkogkas2023} for a particular adaptive Kuramoto-type model \eqref{eq:kuramoto}.  Although the result is only valid for dense graphs and therefore involves working with the graphon framework just as in \cite{AyiPouradierDuteil21,Medvedev14}, it is stated in a slightly more general framework which, if the theory allowed, would make it possible to deal with a more general class of graphs.

We start by recalling some concepts of Graph Theory mentioned in Sections \ref{sec:Intro} and \ref{sec:MFLsparse}. As seen previously, a graphop is a bounded self-adjoint and positivity-preserving operator $A:L^\infty(I) \to L^1(I)$. It can be viewed as a generalized concept for the adjacency matrix. 
Importantly, the Riesz representation theorem allows to define graphops using a family of finite measures $(\eta^\xi)_{\xi \in I}$ called fiber measures as follows:
In the case where $\eta^\xi$ has a density $w$ as $$d\eta^\xi(\zeta) = w(\xi, \zeta)d\zeta,$$ then $w$ plays the role of the usual graphon (see Def. \ref{def:graphon} and section \ref{sec:MFLdense}).
In \cite{gkogkas2023}, the continuum limit of the adaptive Kuramoto-type model \eqref{eq:kuramoto} is proven to be
\begin{equation}
\left\{\begin{array}{l}
\displaystyle \partial_t x(t,\xi) = \omega(\xi,x(t,\xi),t) + \int_I \phi(x(t,\xi),x(t,\zeta)) d \eta_t^\xi(\zeta), \\
\displaystyle \partial_t  \eta_t^\xi(\zeta)= - \varepsilon  \eta_t^\xi(\zeta) + \varepsilon H(x(t,\xi),x(t,\zeta)) \lambda(y)
\end{array}\right.
\label{eq:GraphLimit-kuramoto}
\end{equation} 
where $\lambda$ is the Lebesgue measure. Equation \eqref{eq:GraphLimit-kuramoto} is to be interpreted in the integral form and the equation on $\eta^\xi$ in the weak sense, meaning that for any bounded continuous test-function $f \in \mathcal{C}_b(I)$,
\begin{multline}
\int_I f(\zeta)  d\eta_t^\xi(\zeta) = \int_I f(\zeta)  d\eta_0^\xi(\zeta) - \varepsilon \int_{t_0}^t \left(\int_I f(\zeta)  d\eta_\tau^\xi(\zeta)  \right) d\tau \\
- \varepsilon  \int_{t_0}^t \left( \int_I f(\zeta)  H(x(\tau,\xi),x(\tau,\zeta)) d \zeta\right)d\tau.
\end{multline}
As in the previous case, the proof heavily relies on the regularity assumptions. The functions $\phi,H: \T^2 \to \R$ are Lipschitz continuous and $\omega:I\times \T \times \R \to \R$ is continuous in $t$ and Lipschitz continuous in $x$ and $\xi$. Provided some additional continuity and absolute continuity hypotheses on $(x_0,\eta_0)$ (see \cite[Assumptions (A5)-(A6)]{gkogkas2023}), one can prove the convergence in $\mathcal{C}([t_0,t_0+T], L^\infty(I;\T)) \times \mathcal{C}([t_0,t_0+T], \mathcal{B}(I,\mathcal{M}(I)))$ of the piecewise-constant functions $(x_N,\eta_N)$ defined as 
\begin{equation}
\label{eq:phiN}
\begin{cases}
\displaystyle x_N(t,\xi) := \sum_{i=1}^N x_i^{N}(t) \mathbf{1}_{[\frac{i-1}{N}, \frac{i}{N})}(\xi)\\
\displaystyle w_N(t,\xi, \zeta) := \sum_{i=1}^N w_{ij}^{N}(t) \mathbf{1}_{[\frac{i-1}{N}, \frac{i}{N})}(\xi)  \mathbf{1}_{[\frac{j-1}{N}, \frac{j}{N})}(\zeta)\\
\displaystyle d(\eta_N)_t^\xi(\zeta) := w_N(t,\xi,\zeta) d \zeta
\end{cases}
\end{equation}
to $(x,\eta_t)$ solution to \eqref{eq:GraphLimit-kuramoto} for the appropriate initial data (they are build similarly as for the previous proof) with $\mathcal{B}(I,\mathcal{M}(I))$ being the space of bounded measurable functions from $I$ to $\mathcal{M}(I)$. As before, the existence and uniqueness of solutions relies on the use of a fixed-point theorem while the convergence uses, among other things, the continuous dependence on the initial data and on the function $\omega$ that is established through some stability estimates.

More recently, in the graphon framework, following the spirit of the proofs in \cite{AyiPouradierDuteil23,Medvedev14}, the convergence of \eqref{eq:kuramoto-gen} to the continuum limit equation
\begin{equation}
\left\{\begin{array}{l}
\displaystyle \partial_t x(t,\xi) = \omega(t,\xi,x(t,\cdot)) + \int_I m(t,\zeta)\phi(t, x(t,\xi), x(t,\se)) d\se  \\
\partial_t w(t,\xi,\zeta) = \psi(\xi,\zeta,x(t,\cdot),w(t,\cdot,\cdot)),
\end{array}\right.
\label{eq:GraphLimit_kuramoto-gen}
\end{equation} 
was established under some regularity assumptions and some bounds similar to those of Hypothesis \eqref{hyp:psi}.

\paragraph{Singular weight dynamics.} In \cite{porat2023graph}, the authors proposes to deal with the model \eqref{eq:syst-gen} in the case where the weight dynamics contain a singularity. This approach is motivated by the ``pairwise competition'' model introduced in \cite{McQuadePiccoliPouradierDuteil19}, which can be written as
\begin{equation}\label{eq:singular_weights}
\begin{cases}
\displaystyle \frac{d}{dt} x_i^{N} = \frac{1}{N}  \sum_{j=1}^N m_j^N\, \phi(x_j^{N}-x_i^{N})\\
\displaystyle  \frac{d}{dt} m_i^{N} = \frac{1}{2N^2} m_i^N\sum_{k,j=1}^N   m_k^N  m_j^N \left(\phi(x_k^N-x_i^N)+\phi(x_k^N-x_j^N)\right) \mathbf{s}(x_i^N-x_j^N).
\end{cases}
\end{equation}
The aim of \cite{porat2023graph} is to extend the result established in \cite{AyiPouradierDuteil21} to the case where $\mathbf{s}$ presents a jump discontinuity at the origin. More precisely, one considers the following one-dimensional setting:
\begin{hyp}
The restriction $\mathbf{s}_{|(0,\infty)}$ and  $\mathbf{s}_{|(-\infty,0)}$ are Lipschitz continuous and $\mathbf{s}$ is odd and bounded.
\end{hyp}
Under these assumptions, convergence is obtained to the following system of equations 
\begin{equation}
\left\{\begin{array}{l}
\displaystyle \partial_t x(t,\xi) = \int_I m(t,\zeta)\phi(x(t,\se) - x(t,\xi)) d\se \\
\displaystyle \partial_t m(t,\xi) = m(\xi) \int \int_{I^2} m(\zeta)m(\zeta_*) \left( \phi(x(\zeta_*)-x(\xi))+\phi(x(\xi)-x(\zeta_*))\right) \mathbf{s}(x(\xi)-x(\zeta)) d\zeta d\zeta_*. 
\end{array}\right.
\label{eq:GraphLimit-singular_weights}
\end{equation} 
An important necessary condition for the well-posedness of the microscopic dynamics is the fact that the opinions remain separated for all times provided that they are separated initially. Similarly, the solution to the continuum limit equation is expected to satisfy an analogue property: in dimension 1, this is contained in the assumption that $x^0$ is one-to-one. Then, a pointwise evaluation of $x^0$ is needed and therefore, unlike the previous frameworks, a more natural assumption is $x^0 \in \mathcal{C}(I)$ rather than $x^0 \in L^\infty(I)$.

The proofs for the existence and uniqueness of solutions to \eqref{eq:GraphLimit-singular_weights} and of the convergence to the continuum limit equation of \eqref{eq:singular_weights} are quite similar to those in \cite{AyiPouradierDuteil21}, the main difference being of course in the handling of the equation for the weight dynamics. Because of the singularity in $0$, one introduces the sets $$A_N(t,\xi) :=\{\zeta \in I ~|~ x_N(t,\zeta)-x_N(\xi,t) >0\}, \quad B_N(t,\xi):=A_N^c(t,\xi)$$ and $$A(t,\xi) :=\{\zeta \in I ~|~ x(t,\zeta)-x(\xi,t) >0\}, \quad B(t,\xi):=A^c(t,\xi)$$
and uses them to split the integrals appearing in the computations into several terms. One can then conclude using the convergence of $\mathbf{1}_{A_N(0,\xi)}(\zeta)$ to $\mathbf{1}_{A(0,\xi)}(\zeta)$ and of $\mathbf{1}_{B_N(0,\xi)}(\zeta)$ to $\mathbf{1}_{B(0,\xi)}(\zeta)$ as $N \to \infty$. 
\begin{rem} This result is also also extended to any dimension $d>1$ with $\xi$ varying in the $d$-dimensional unit cube $I^d$ but we will not expand on this point and refer the reader to \cite[Section 5]{porat2023graph} for more details.
\end{rem}


\subsection{Links between Continuum limit and Mean-Field limits for the case of evolving-in-time weights}
\label{Links2}

This section is devoted to bridging all the different limit equations obtained for the opinion model with time-evolving weights. 
As previously, for $N \in \N$, we denote $x_i^N(t) \in \R^d$ the opinion of agent $i$ at time $t$ and $m_i^N(t) \in \R$ its weight (representing its charisma). We  are interested in the system
\begin{equation}\label{eq:syst-gen2}
\begin{cases}
\displaystyle \frac{d}{dt} x_i^{N}(t) = \frac{1}{N}  \sum_{j=1}^N m_j^N(t)\, \phi(x_i^{N}(t),x_j^{N}(t))\\
\displaystyle  \frac{d}{dt} m_i^{N}(t) = \psi_i^{(N)}(x^{N}(t),m^{N}(t)).
\end{cases}
\end{equation}
For initial data such that $m_0$ is strictly positive, one can actually prove that $m^N(t)$ remains positive for all $t \in [0,T]$. Thus, from here onwards, the weights are assumed to belong to $\R_+^*$. Taking the continuum limit leads to the convergence in $ \mathcal{C}([0,T];L^2(I;\R^{dN})) \times \mathcal{C}([0,T];L^2(I;(\R_+^*)^{N}))$ to $(x,m)$ which is a solution to the integro-differential equation
\begin{equation}
\left\{\begin{array}{l}
\displaystyle \partial_t x(t,\xi) = \int_I m(t,\zeta)\phi(x(t,\xi), x(t,\se)) d\se \\
\partial_t m(t,\xi) = \psi(\xi,x(t,\cdot),m(t,\cdot)),
\end{array}\right.
\label{eq:GraphLimit-gen2}
\end{equation} 
provided that the initial data converges in a suitable sense.

\paragraph{The non exchangeable mean-field limit.} As mentioned  in Subsection \ref{subsec:non_exchangeable_mfl}, for weight dynamics of the form
\begin{equation*}
\psi_i(x^N,m^N) = \int_{\{1,\dots,N\} \times \R^d \times \R_+^*} \tilde{S}_i(x_i^N,m_i^N,k,y,n)d\tilde{\mu}_t^N(k,y,n) 
\end{equation*}
with the appropriate regularity assumptions, the non-exchangeable mean-field limit satisfied by the limit probability density $\mu \in \mathcal{C}([0,T], \mathcal{P}(I \times \R^d \times \R_+^*))$ is
\begin{multline}\label{eq:non_exchangeable_mfl2}
\displaystyle \partial_t \mu_t^\xi(x,m) + \nabla_x\cdot \left( \left(\int_{I\times\R^d \times \R_+^*}  n \phi(x,y) d\mu_t^\zeta(y,n)d\zeta\right)  \mu_t^\xi(x,m) \right) \\
+ \partial_m \left( \left( \int_{I\times\R^d \times \R_+^*} \tilde{S}(\xi,x,m,\zeta,y,n)  d\mu_t^\zeta(y,n)d\zeta \right)  \mu_t^\xi(x,m) \right) = 0 
\end{multline}

\paragraph{The exchangeable mean-field limits.}  In the case where the agents are indistinguishable, in the sense mentioned in Section \ref{mfl_adaptive}, we can be interested in two different probability densities: the one for weights and opinions belonging to $\mathcal{C}([0,T], \mathcal{P}(\R^d \times \R_+^*))$ and the one only for opinions belonging to $\mathcal{C}([0,T], \mathcal{P}(\R^d))$. In both cases, the mean-field limit can actually be obtained for a class, slightly different from the one presented previously, with a straightforward adaptation of the proof. \\

\begin{enumerate}
    \item  For weight dynamics of the form
\begin{equation}
\label{eq:wd_exchangeable1}
\psi(\xi,x(t,\cdot),m(t,\cdot)) = \int_{\R^d \times \R_+^*} \overline{S}(x(t,\xi),m(t,\xi),y,n)  d\overline{\mu}_t(y,n) 
\end{equation}
with \begin{equation}
\label{eq:empirical_measure_exchangeable1}
\overline{\mu}_t(x,m) =  \int_I \delta({x-x(t,\xi)}) \delta({m-m(t,\xi)}) d\xi,
\end{equation}
the classical mean-field equation can be derived for the probability measure $\mu \in \mathcal{C}([0,T], \mathcal{P}(\R^d \times \R_+^*))$, obtaining a transport equation in the variables $x$ and $m$: 
\begin{multline}\label{eq:exchangeable_mfl1}
\displaystyle \partial_t \mu_t(x,m) + \nabla_x\cdot \left( \left(\int_{\R^d \times \R_+^*}  n \phi(x,y) d\mu_t(y,n)\right)  \mu_t(x,m) \right) \\
+ \partial_m \left( \left( \int_{\R^d \times \R_+^*} \overline{S}(x,m,y,n)  d\mu_t(y,n)\right)  \mu_t(x,m) \right) = 0 
\end{multline}

\item As discussed in Section \ref{mfl_adaptive}, for the  general class of weight dynamics
\begin{equation}
\label{eq:wd_exchangeable2}
\psi(\xi,x(t,\cdot),m(t,\cdot)) = m(t,\xi) \int_{\R^d } \hat{S}(x(t,\xi),y)  d\hat{\mu}_t(y) 
\end{equation}
with \begin{equation}
\label{eq:empirical_measure_exchangeable2}
\hat{\mu}_t(x) =  \int_I m(t,\xi) \delta({x-x(t,\xi)}) d\xi,
\end{equation}
the probability measure $\mu \in \mathcal{C}([0,T], \mathcal{P}(\R^d))$ obtained in the limit satisfies the following transport equation with source: 
\begin{equation}
\label{eq:exchangeable_mfl2}
\partial_t \mu_t(x) + \nabla_x\cdot \left( \int_{\R^d} \phi(x,y) d\mu(y)\mu_t(x)\right) = \mu_t(x) \int_{\R^d } \hat{S}(x,y)  d\hat{\mu}_t(y).
\end{equation}
\end{enumerate}

Thus, \eqref{eq:GraphLimit-gen2}, \eqref{eq:non_exchangeable_mfl2}, \eqref{eq:exchangeable_mfl1},  \eqref{eq:exchangeable_mfl2} can be obtained as the limit of \eqref{eq:syst-gen2} as $N$ goes to infinity. The derivations of \eqref{eq:GraphLimit-gen2} and  \eqref{eq:exchangeable_mfl2} have been done rigorously in \cite{AyiPouradierDuteil21}. The derivation of  \eqref{eq:non_exchangeable_mfl} can be obtained thanks to the results in \cite{PT23}. We will say a word about the derivation of \eqref{eq:exchangeable_mfl1} in the following. As in section \ref{sec:Link}, let us try to clarify how all theses different descriptions relate to each other (see Figure \ref{schema} for a visual summary).

\begin{figure}
\begin{center}
\scalebox{0.8}{
\begin{tikzpicture}[>=latex,font=\sffamily]
\tikzstyle{block1} = [draw, rectangle, text width=18em, text centered, minimum height=4em, rounded corners]
\tikzstyle{block2} = [draw, rectangle, text width=19em, text centered, minimum height=4em, rounded corners]
\tikzstyle{block3} = [draw, rectangle, text width=25.6em, text centered, minimum height=4em, rounded corners]
\tikzstyle{block4} = [draw, rectangle, text width=26em, text centered, minimum height=4em, rounded corners]
\tikzstyle{block5} = [draw, rectangle, text width=16.3em, text centered, minimum height=4em, rounded corners]
\tikzstyle{arrow} = [->, thick]
\tikzstyle{arrow1} = [->, thick,blue,loosely dashed]
\tikzstyle{arrow2} = [->, thick,red]
\tikzstyle{arrow3} = [->, thick,densely dotted]

\node [block1] (A1) at (0,0) {\textbf{The system of ODEs} \small \begin{equation*}
\begin{cases}
\displaystyle \frac{d}{dt} x_i^{N}(t) = \frac{1}{N}  \sum_{j=1}^N m_j^N(t)\, \phi(x_i^{N}(t),x_j^{N}(t))\\
\displaystyle  \frac{d}{dt} m_i^{N}(t) = \psi_i^{(N)}(x^{N}(t),m^{N}(t))
\end{cases}
\end{equation*}};
\node [block2] (A2) at (10,0) {\textbf{The continuum limit equation} \small  \begin{equation*}
\left\{\begin{array}{l}
\displaystyle \partial_t x(t,\xi) = \int_I m(t,\zeta)\phi(x(t,\xi), x(t,\se)) d\se \\
\partial_t m(t,\xi) = \psi(\xi,x(t,\cdot),m(t,\cdot))
\end{array}\right.
\end{equation*} };
\node [block3] (B1) at (0,-6) {\textbf{The non-exchangeable mean-field limit equation} \small \begin{multline*}
\displaystyle \partial_t \mu_t^\xi(x,m) + \nabla_x\cdot \left( \left(\int_{I\times\R^d \times \R_+^*}  n \phi(x,y) d\mu_t^\zeta(y,n)d\zeta\right)  \mu_t^\xi(x,m) \right) \\
+ \partial_m \left( \left( \int_{I\times\R^d \times \R_+^*} \tilde{S}(\xi,x,m,\zeta,y,n)  d\mu_t^\zeta(y,n)d\zeta \right)  \mu_t^\xi(x,m) \right) = 0 
\end{multline*}};
\node [block4] (B2) at (4,-12) {\textbf{The exchangeable mean-field limit equation for opinions and weights} \begin{multline*}
\displaystyle \partial_t \mu_t(x,m) + \nabla_x\cdot \left( \left(\int_{\R^d \times \R_+^*}  n \phi(x,y) d\mu_t(y,n)\right)  \mu_t(x,m) \right) \\
+ \partial_m \left( \left( \int_{\R^d \times \R_+^*} \overline{S}(x,m,y,n)  d\mu_t(y,n)\right)  \mu_t(x,m) \right) = 0 
\end{multline*}};
\node [block5] (B3) at (10.8,-6) {\textbf{The exchangeable mean-field limit equation for opinions} 
\begin{equation*}
\begin{array}{l}
\displaystyle\partial_t \mu_t(x) + \nabla\cdot ( \int_{\R^d} \phi(x,y) d\mu(y)\mu_t(x))\\ \displaystyle~~~~~~= \mu_t(x) \int_{\R^d } \hat{S}(x,y)  d\hat{\mu}_t(y) 
\end{array}
\end{equation*}};

\draw [arrow2] (A1) -- node[above] {$N \to \infty$ }  (A2);
\draw [arrow2] (A1) -- node[left] {$N \to \infty$ } (B1);
\draw [arrow, bend left=10,red] (A1) to node[right]{$N \to \infty$} (B3);
\draw [arrow1] (A2) -- node[left] {\circled{3}} (B3);
\draw [arrow1] (A2) -- node[left] {\circled{1}} (B1);
\draw [arrow1] (A2) -- node[left] {\circled{2}} (B2);
\draw [arrow3] (B1) -- node[left] {\circled{4}} (B2);
\draw [arrow3] (B2) -- node[right] {\circled{5}} (B3);
\draw [arrow, bend left=40,red] (A1) to node[right] {$N \to \infty$ } (B2);
\end{tikzpicture}
}
\end{center}
\caption{Links between the different equations. The arrow 1 corresponds to the weight dynamics $\displaystyle \psi_1(\xi,x(t,\cdot),m(t,\cdot)) =  \int_{I\times\R^d \times \R_+^*} \tilde{S}(\xi,x(t,\xi),m(t,\xi),\zeta,y,n)  d\tilde{\mu}_t^\zeta(y,n)d\zeta$, the arrow 2 to the weight dynamics $\displaystyle
\psi_2(\xi,x(t,\cdot),m(t,\cdot)) = \int_{\R^d \times \R_+^*} \overline{S}(x(t,\xi),m(t,\xi),y,n)  d\overline{\mu}_t(y,n) 
$  and the arrow 3 to  the weight dynamics $\displaystyle \psi_3(\xi,x(t,\cdot),m(t,\cdot)) = m(t,\xi) \int_{\R^d } \hat{S}(x(t,\xi),y)  d\hat{\mu}_t(y)$. The arrows 4 and 5 correspond to Remark \ref{arrow4_5}.\label{schema}}
\end{figure}
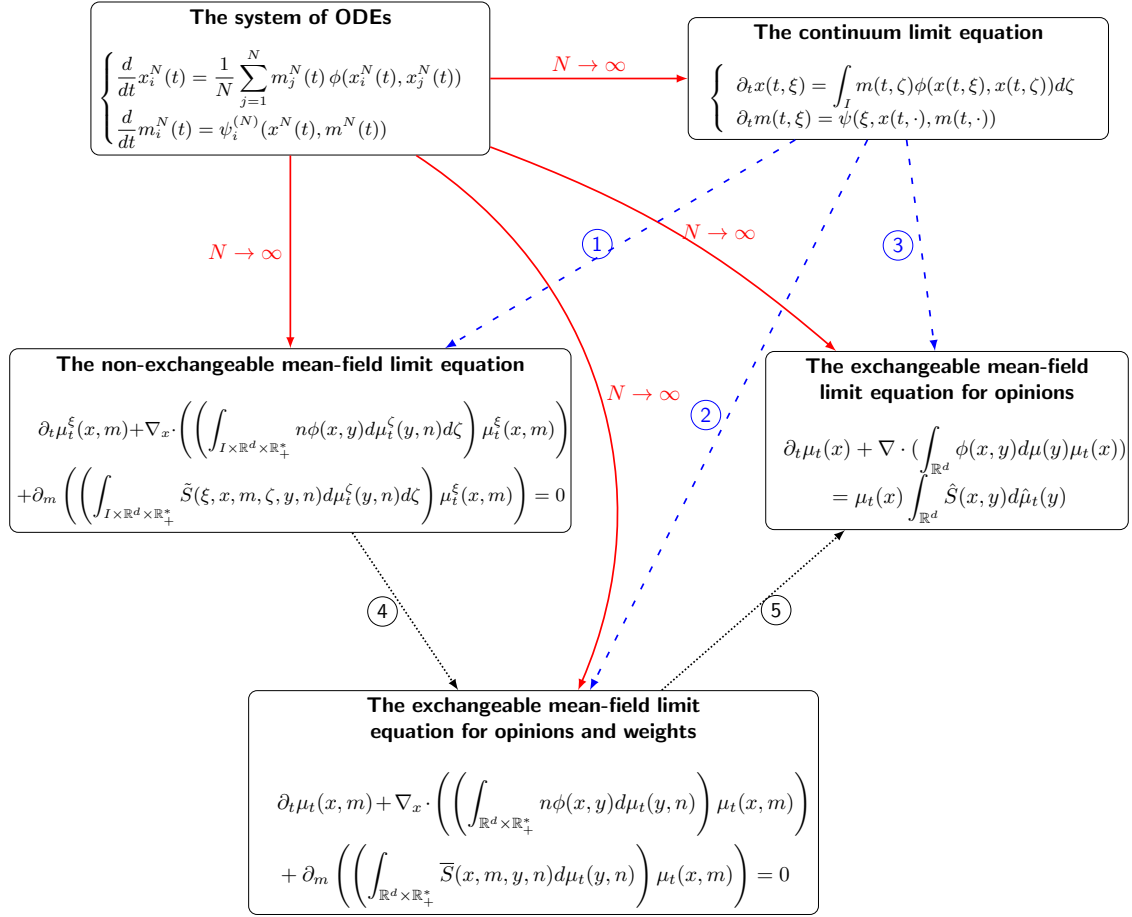

\subsubsection{From continuum limit to non-exchangeable mean-field limit}
We denote $\tilde{\mu}_t^\xi$ the following continuous empirical measure
\begin{equation}
\tilde{\mu}_t^\xi(x,m) =  \int_I \delta({x-x(t,\zeta)}) \delta({m-m(t,\zeta)})  \delta({\xi-\zeta})d\zeta.
\end{equation}
where $(x(t,\xi),m(t,\xi))$ is a solution to \eqref{eq:GraphLimit-gen2} for the weight dynamics 
\begin{equation}
\label{eq:wd_non_exchangeable}
\psi(\xi,x(t,\cdot),m(t,\cdot)) =  \int_{I\times\R^d \times \R_+^*} \tilde{S}(\xi,x(t,\xi),m(t,\xi),\zeta,y,n)  d\tilde{\mu}_t^\zeta(y,n)d\zeta .
\end{equation}
For all test functions $f\in C^\infty(I\times\Rd\times \R_+^*)$, it holds
\[
\begin{split}
& \frac{d}{dt}\int_{I\times\Rd \times \R_+^*} f(\xi,x,m) d\tmu_t^\xi(x,m) d\xi 
= \frac{d}{dt}\int_{I} f(\xi,x(t,\xi),m(t,\xi)) d\xi \\
= \quad & \int_{I} \nabla_xf(\xi,x(t,\xi),m(t,\xi)) \cdot \left( \int_I m((t,\zeta)\phi(x(t,\xi),x(t,\zeta)) d\zeta \right)  d\xi \\
&~~~~~~~~~~~ + \int_I \partial_m f(\xi,x(t,\xi),m(t,\xi))  \left(\int_{I\times\Rd \times \R_+^*} \tilde{S}(\xi,x(t,\xi),m(t,\xi),\zeta,y,n)  d\tilde{\mu}_t^\zeta(y,n)d\zeta \right) d\xi \\
= \quad & \int_{I\times\Rd \times \R_+^*}   \nabla_xf(\xi,x,m) \cdot \left( \int_{I\times\Rd \times \R_+^*} n \phi(x,y) d\tilde{\mu}_t^\zeta(y,n)d\zeta \right)  d\tilde{\mu}_t^\xi(x,m)  d\xi \\
&~~~~~~~~~~~ + \int_{I\times\Rd \times \R_+^*}  \partial_m f(\xi,x,m)  \left(\int_{I\times\Rd \times \R_+^*} \tilde{S}(\xi,x,m,\zeta,y,n)  d\tilde{\mu}_t^\zeta(y,n)d\zeta  \right) d\tilde{\mu}_t^\xi(x,m)  d\xi \\
\end{split}
\]
Thus, the continuous empirical measure  $\tilde{\mu}_t^\xi$  built from the solution to \eqref{eq:GraphLimit-gen2} for the weight dynamics \eqref{eq:wd_non_exchangeable} is a solution to the non-exchangeable mean-field limit equation \eqref{eq:non_exchangeable_mfl}. This computation corresponds to the arrow 1 in Figure \ref{schema}.

\subsubsection{From continuum limit to  exchangeable mean-field limits}

For weight dynamics of the form \eqref{eq:wd_exchangeable1}, similar computations against test functions $f\in C^\infty(\Rd\times \R_+^*)$ show that the continuous empirical measure $\overline{\mu}_t(x,m)$ defined in \eqref{eq:empirical_measure_exchangeable1} satisfies the exchangeable mean-field equation for opinions and weights \eqref{eq:exchangeable_mfl1}. Indeed, we get, for all test functions $f\in C^\infty(\Rd\times \R_+^*)$, 
\[
\begin{split}
& \frac{d}{dt}\int_{\Rd \times \R_+^*} f(x,m) d\overline{\mu}_t(x,m) d\xi 
= \frac{d}{dt}\int_{I} f(x(t,\xi),m(t,\xi)) d\xi \\
= \quad & \int_{I} \nabla_xf(x(t,\xi),m(t,\xi)) \cdot \left( \int_I m((t,\zeta)\phi(x(t,\xi),x(t,\zeta)) d\zeta \right)  d\xi \\
&~~~~~~~~~~~ + \int_I \partial_m f(x(t,\xi), m(t,\xi))  \left(\int_{\Rd \times \R_+^*} \overline{S}(x(t,\xi),m(t,\xi),y,n)  d\overline{\mu}_t(y,n)\right) d\xi \\
= \quad & \int_{\Rd \times \R_+^*} \nabla_xf(x,m) \cdot \left( \int_{\Rd \times \R_+^*} n \phi(x,y) d\overline{\mu}_t(y,n) \right)  d\overline{\mu}(x,m)\\
&~~~~~~~~~~~ + \int_{\Rd \times \R_+^*}  \partial_m f(x,m)  \left(\int_{\Rd \times \R_+^*} \overline{S}(x,m,y,n)  d\overline{\mu}_t(y,n) \right) d\overline{\mu}_t(x,m).  \\
\end{split}
\]
This computation corresponds to the arrow 2 in Figure \ref{schema}.

Lastly, similarly to \cite{AyiPouradierDuteil21}, for weight dynamics of the form \eqref{eq:wd_exchangeable2},  the continuous empirical measure $\hat{\mu}_t(x)$ defined in \eqref{eq:empirical_measure_exchangeable2} satisfies the exchangeable mean-field equation for opinions \eqref{eq:exchangeable_mfl2}. Indeed, for all test functions $f\in C^\infty(\Rd)$, 
\[
\begin{split}
& \frac{d}{dt}\int_{\Rd } f(x) d\hat{\mu}_t(x) d\xi 
= \frac{d}{dt}\int_{I} m(t,\xi) f(x(t,\xi)) d\xi \\
= \quad & \int_{I}  m(t,\xi)   \nabla_xf(x(t,\xi)) \cdot \left( \int_I m((t,\zeta)\phi(x(t,\xi),x(t,\zeta)) d\zeta \right)  d\xi \\
&~~~~~~~~~~~ + \int_I  m(t,\xi)   f(x(t,\xi))  \left(\int_{\Rd} \hat{S}(x(t,\xi),y)  d\hat{\mu}_t(y) \right) d\xi \\
= \quad & \int_{\Rd } \nabla_xf(x) \cdot \left( \int_{\Rd }  \phi(x,y) d\hat{\mu}_t(y) \right)  d\hat{\mu}(x)\\
&~~~~~~~~~~~ + \int_{\Rd }  f(x)  \left(\int_{\Rd} \hat{S}(x,y)  d\hat{\mu}_t(y) \right) d\hat{\mu}_t(x).  \\
\end{split}
\]
This computation corresponds to the arrow 3 in Figure \ref{schema}.
\begin{rem}
\label{arrow4_5}
We notice that $$\overline{\mu}_t(x,m) = \int_I \tilde{\mu}^\xi(x,m) d\xi$$ and $$\hat{\mu}_t(x) = \int_\R m \overline{\mu}_t(x,dm).$$
Thus, for the appropriate choice for the weight dynamics, this gives  the path to obtain \eqref{eq:exchangeable_mfl1} from \eqref{eq:non_exchangeable_mfl} and \eqref{eq:exchangeable_mfl2} from  \eqref{eq:exchangeable_mfl1} and corresponds to the arrows 4 and 5 in Figure \ref{schema}.
\end{rem}

\begin{rem}
Because of the presence of  nonlinearity, we cannot adapt the argument developed in Section \ref{non_exchangeable_to_graph_limit} and we cannot exhibit a path to get the continuum limit \eqref{eq:GraphLimit-gen2} from the non-exchangeable mean-field limit \eqref{eq:non_exchangeable_mfl}.
\end{rem}

The continuum limit can actually provide a tool for an alternative proof for the exchangeable mean-field limit. Indeed, noticing that we can rewrite $\hat{\mu}_t^N$ defined in \eqref{eq:empmes} as 
$$\hat{\mu}_t^N= \int_I m_N(t,\xi)  \delta({x-x_N(t,\xi)}) $$ with $(x_N,m_N)$ defined in \eqref{eq:xN}, we can easily show that $W_1(\hat{\mu}_t^N, \hat{\mu}_t) \xrightarrow[N\rightarrow+\infty]{} 0$ for all $t \in [0,T]$ provided the appropriate choice for the initial data using the convergence of $(x_N,m_N)$ to $(x,m)$ with $W_1$ being the $1$-Wasserstein distance. Moreover, we just showed that $\hat{\mu}_t$ happens to be a solution to \eqref{eq:exchangeable_mfl2}, hence the conclusion. 

All the results mentioned in that section require some regularity assumptions on both the interaction function and the weight dynamics. However, in \cite{porat2023graph}, by similar computations, they are able to obtain the derivation of the exchangeable  mean-field limit for the probability density of opinions associated to their singular weight dynamics. It is worth noticing that the continuum limit turns out to be a powerful tool since, in that case, there is no alternative proof through the classical approach. Indeed, all the classical approaches use some regularity of the source term which is no longer true for singular weight dynamics. 

Finally, noticing that can rewrite $\overline{\mu}_t^N$ defined in \eqref{eq:empirical_measure_exchangeable1} as 
$$\overline{\mu}_t^N= \int_I \delta({x-x_N(t,\xi)})  \delta({m-m_N(t,\xi)}),$$ we can similarly prove that  $W_1(\overline{\mu}_t^N, \overline{\mu}_t) \xrightarrow[N\rightarrow+\infty]{} 0$ for all $t \in [0,T]$ provided the appropriate choice for the initial data. Thus, this provides the derivation of \eqref{eq:exchangeable_mfl1} from \eqref{eq:syst-gen2} as $N$ goes to infinity.

\bibliographystyle{abbrv}
\bibliography{Biblio_Review}

\begin{thebibliography}{10}

\bibitem{Aoki82}
I.~Aoki.
\newblock A simulation study on the schooling mechanism in fish.
\newblock {\em NIPPON SUISAN GAKKAISHI}, 48(8):1081--1088, 1982.

\bibitem{AyiPouradierDuteil21}
N.~Ayi and N.~{Pouradier Duteil}.
\newblock Mean-field and graph limits for collective dynamics models with
  time-varying weights.
\newblock {\em Journal of Differential Equations}, 299:65--110, 2021.

\bibitem{AyiPouradierDuteil23}
N.~Ayi and N.~Pouradier~Duteil.
\newblock Graph limit for interacting particle systems on weighted random
  graphs.
\newblock {\em arXiv preprint arXiv:2307.12801}, 2023.

\bibitem{BackhauszSzegedy22}
A.~Backhausz and B.~Szegedy.
\newblock Action convergence of operators and graphs.
\newblock {\em Canadian Journal of Mathematics}, 74(1):72–121, 2022.

\bibitem{Ballerini08}
M.~Ballerini, N.~Cabibbo, R.~Candelier, A.~Cavagna, E.~Cisbani, I.~Giardina,
  V.~Lecomte, A.~Orlandi, G.~Parisi, A.~Procaccini, et~al.
\newblock Interaction ruling animal collective behavior depends on topological
  rather than metric distance: Evidence from a field study.
\newblock {\em Proceedings of the national academy of sciences},
  105(4):1232--1237, 2008.

\bibitem{BDPZ20}
J.~Barré, P.~Degond, D.~Peurichard, and E.~Zatorska.
\newblock Modelling pattern formation through differential repulsion.
\newblock {\em Networks and Heterogeneous Media}, 15(3):307--352, 2020.

\bibitem{BenPorat2023}
I.~Ben-Porat, J.~A. Carrillo, and S.~T. Galtung.
\newblock Mean field limit for one dimensional opinion dynamics with coulomb
  interaction and time dependent weights.
\newblock {\em Nonlinear Analysis}, 240:113462, Mar. 2024.

\bibitem{rev_kuehn}
R.~Berner, T.~Gross, C.~Kuehn, J.~Kurths, and S.~Yanchuk.
\newblock Adaptive dynamical networks.
\newblock {\em Physics Reports}, 1031:1--59, 2023.
\newblock Adaptive dynamical networks.

\bibitem{BetCoppiniNardi23}
G.~Bet, F.~Coppini, and F.~R. Nardi.
\newblock Weakly interacting oscillators on dense random graphs.
\newblock {\em Journal of Applied Probability}, page 1–24, 2023.

\bibitem{BhamidiBudhijaratWu19}
S.~Bhamidi, A.~Budhiraja, and R.~Wu.
\newblock Weakly interacting particle systems on inhomogeneous random graphs.
\newblock {\em Stochastic Processes and their Applications}, 129(6):2174--2206,
  2019.

\bibitem{BiccariKoZuazua19}
U.~Biccari, D.~Ko, and E.~Zuazua.
\newblock {Dynamics and control for multi-agent networked systems: A
  finite-difference approach}.
\newblock {\em {Mathematical Models and Methods in Applied Sciences}},
  29(04):755--790, Apr. 2019.

\bibitem{BonnetPouradierDuteilSigalotti21}
B.~Bonnet, N.~Pouradier~Duteil, and M.~Sigalotti.
\newblock Consensus formation in first-order graphon models with time-varying
  topologies.
\newblock {\em Mathematical Models and Methods in Applied Sciences},
  32(11):2121--2188, 2022.

\bibitem{BorgsChayesCohnZhao19}
C.~Borgs, J.~Chayes, H.~Cohn, and Y.~Zhao.
\newblock An $l^{p}$ theory of sparse graph convergence i: Limits, sparse
  random graph models, and power law distributions.
\newblock {\em Transactions of the American Mathematical Society},
  372(5):3019--3062, 2019.

\bibitem{BST22}
L.~Boudin, F.~Salvarani, and E.~Tr\'{e}lat.
\newblock Exponential convergence towards consensus for non-symmetric linear
  first-order systems in finite and infinite dimensions.
\newblock {\em SIAM Journal on Mathematical Analysis}, 54(3):2727--2752, 2022.

\bibitem{braun1977}
W.~Braun and K.~Hepp.
\newblock {{The Vlasov dynamics and its fluctuations in the $1/N$ limit of
  interacting classical particles}}.
\newblock {\em Comm. Math. Phys.}, 56(2):101--113, 1977.

\bibitem{BraunHepp77}
W.~Braun and K.~Hepp.
\newblock The vlasov dynamics and its fluctuations in the 1/n limit of
  interacting classical particles.
\newblock {\em Communications in mathematical physics}, 56(2):101--113, 1977.

\bibitem{BulloCortezMartinez09}
F.~Bullo, J.~Cort{\'e}s, and S.~Martinez.
\newblock {\em Distributed control of robotic networks: a mathematical approach
  to motion coordination algorithms}, volume~27.
\newblock Princeton University Press, 2009.

\bibitem{ChibaMedvedev19}
H.~Chiba and G.~S. Medvedev.
\newblock The mean field analysis of the {K}uramoto model on graphs i. the mean
  field equation and transition point formulas.
\newblock {\em Discrete and Continuous Dynamical Systems}, 39(1):131--155,
  2019.

\bibitem{CS07}
F.~Cucker and S.~Smale.
\newblock Emergent behavior in flocks.
\newblock {\em IEEE Transactions on Automatic Control}, 52:852--862, 2007.

\bibitem{DelattreGiacominLucon16}
S.~Delattre, G.~Giacomin, and E.~Lu{\c{c}}on.
\newblock A note on dynamical models on random graphs and fokker--planck
  equations.
\newblock {\em Journal of Statistical Physics}, 165:785--798, 2016.

\bibitem{Dobrushin79}
R.~L. Dobrushin.
\newblock {Vlasov equations}.
\newblock {\em Functional Analysis and Its Applications}, 13(2):115--123, 1979.

\bibitem{FaureMaury15}
S.~Faure and B.~Maury.
\newblock Crowd motion from the granular standpoint.
\newblock {\em Mathematical Models and Methods in Applied Sciences},
  25(03):463--493, 2015.

\bibitem{GkogkasKuehn22}
M.~A. Gkogkas and C.~Kuehn.
\newblock Graphop mean-field limits for {K}uramoto-type models.
\newblock {\em SIAM Journal on Applied Dynamical Systems}, 21(1):248--283,
  2022.

\bibitem{gkogkas2023}
M.~A. Gkogkas, C.~Kuehn, and C.~Xu.
\newblock Continuum limits for adaptive network dynamics.
\newblock {\em Communication in Mathematical Sciences}, 21:83--106, 2023.

\bibitem{gkogkas2023mean}
M.~A. Gkogkas, C.~Kuehn, and C.~Xu.
\newblock Mean field limits of co-evolutionary heterogeneous networks, 2023.

\bibitem{Golse16}
F.~Golse.
\newblock On the dynamics of large particle systems in the mean field limit.
\newblock {\em Macroscopic and large scale phenomena: coarse graining, mean
  field limits and ergodicity}, pages 1--144, 2016.

\bibitem{HaKimZhang18}
S.-Y. Ha, J.~Kim, and X.~Zhang.
\newblock Uniform stability of the cucker-smale model and its application to
  the mean-field limit.
\newblock {\em Kinetic and Related Models}, 11(5):1157--1181, 2018.

\bibitem{HaLiu09}
S.-Y. Ha and J.-G. Liu.
\newblock {A simple proof of the Cucker-Smale flocking dynamics and mean-field
  limit}.
\newblock {\em Communications in Mathematical Sciences}, 7(2):297 -- 325, 2009.

\bibitem{HK}
R.~Hegselmann and U.~Krause.
\newblock {Opinion Dynamics and Bounded Confidence Models, Analysis and
  Simulation}.
\newblock {\em Journal of Artificial Societies and Social Simulation}, 5, 07
  2002.

\bibitem{JabinPoyatoSoler21}
P.-E. Jabin, D.~Poyato, and J.~Soler.
\newblock {Mean-field limit of non-exchangeable systems}.
\newblock working paper or preprint, Mar. 2022.

\bibitem{JabinWang18}
P.-E. Jabin and Z.~Wang.
\newblock Quantitative estimates of propagation of chaos for stochastic systems
  with $w^{-1,\infty}$ kernels.
\newblock {\em Inventiones mathematicae}, 214:523--591, 2018.

\bibitem{JadbabaieMoteeBarahona04}
A.~Jadbabaie, N.~Motee, and M.~Barahona.
\newblock On the stability of the kuramoto model of coupled nonlinear
  oscillators.
\newblock In {\em Proceedings of the 2004 American Control Conference},
  volume~5, pages 4296--4301 vol.5, 2004.

\bibitem{KaliuzhnyiMedvedev18}
D.~Kaliuzhnyi-Verbovetskyi and G.~S. Medvedev.
\newblock The mean field equation for the {K}uramoto model on graph sequences
  with non-lipschitz limit.
\newblock {\em SIAM Journal on Mathematical Analysis}, 50(3):2441--2465, 2018.

\bibitem{Kuehn20}
C.~Kuehn.
\newblock Network dynamics on graphops.
\newblock {\em New Journal of Physics}, 22(5):053030, may 2020.

\bibitem{KuehnXu22}
C.~Kuehn and C.~Xu.
\newblock Vlasov equations on digraph measures.
\newblock {\em Journal of Differential Equations}, 339:261--349, 2022.

\bibitem{KunszentiLovaszSzegedy19}
D.~Kunszenti-Kov{\'a}cs, L.~Lov{\'a}sz, and B.~Szegedy.
\newblock Measures on the square as sparse graph limits.
\newblock {\em Journal of Combinatorial Theory, Series B}, 138:1--40, 2019.

\bibitem{Kuramoto75}
Y.~Kuramoto.
\newblock Self-entrainment of a population of coupled nonlinear oscillators.
\newblock In H.~Araki, editor, {\em International Symposium on Mathematical
  Problems in Theoretical Physics, Lecture Notes in Physics, Vol. 39,}, pages
  420--422, New York, NY, USA, 1975. Springer.

\bibitem{LeCun}
Y.~LeCun, Y.~Bengio, and G.~Hinton.
\newblock Deep learning.
\newblock {\em Nature}, 521:436--44, 05 2015.

\bibitem{Lopez12}
U.~Lopez, J.~Gautrais, I.~D. Couzin, and G.~Theraulaz.
\newblock From behavioural analyses to models of collective motion in fish
  schools.
\newblock {\em Interface focus}, 2(6):693--707, 2012.

\bibitem{Lovasz12}
L.~Lovász.
\newblock {\em Large Networks and Graph Limits.}, volume~60 of {\em Colloquium
  Publications}.
\newblock American Mathematical Society, 2012.

\bibitem{LovaszSzegedy06}
L.~Lovász and B.~Szegedy.
\newblock Limits of dense graph sequences.
\newblock {\em Journal of Combinatorial Theory, Series B}, 96(6):933--957,
  2006.

\bibitem{McQuadePiccoliPouradierDuteil19}
S.~McQuade, B.~Piccoli, and N.~Pouradier~Duteil.
\newblock {Social dynamics models with time-varying influence}.
\newblock {\em Math. Models Methods Appl. Sci.}, 29(4):681--716, 2019.

\bibitem{medvedev_sw}
G.~S. Medvedev.
\newblock Small-world networks of kuramoto oscillators.
\newblock {\em Physica D: Nonlinear Phenomena}, 266:13--22, 2014.

\bibitem{Medvedev14}
G.~S. Medvedev.
\newblock {The Nonlinear Heat Equation on Dense Graphs and Graph Limits}.
\newblock {\em SIAM J. Math. Analysis}, 46:2743--2766, 2014.

\bibitem{MedvedevR}
G.~S. Medvedev.
\newblock {The Nonlinear Heat Equation on W-Random Graphs}.
\newblock {\em Archive for Rational Mechanics and Analysis}, 2014.

\bibitem{Neunzert78}
H.~Neunzert.
\newblock Mathematical investigations on particle-in-cell methods.
\newblock {\em Fluid Dyn. Trans}, 9:229--254, 1978.

\bibitem{Neunzert84}
H.~Neunzert.
\newblock An introduction to the nonlinear boltzmann-vlasov equation.
\newblock In C.~Cercignani, editor, {\em Kinetic Theories and the Boltzmann
  Equation}, pages 60--110, Berlin, Heidelberg, 1984. Springer Berlin
  Heidelberg.

\bibitem{NeunzertWick74}
H.~Neunzert and J.~Wick.
\newblock Die approximation der l{\"o}sung von integro-differentialgleichungen
  durch endliche punktmengen.
\newblock In R.~Ansorge and W.~T{\"o}rnig, editors, {\em Numerische Behandlung
  nichtlinearer Integrodifferential-und Differentialgleichungen}, pages
  275--290, Berlin, Heidelberg, 1974. Springer Berlin Heidelberg.

\bibitem{Nishikawa_2015}
T.~Nishikawa and A.~E. Motter.
\newblock Comparative analysis of existing models for power-grid
  synchronization.
\newblock {\em New Journal of Physics}, 17(1):015012, jan 2015.

\bibitem{nugent2023evolving}
A.~J. Nugent, S.~N. Gomes, and M.-T. Wolfram.
\newblock On evolving network models and their influence on opinion formation,
  2023.

\bibitem{PT23}
T.~Paul and E.~Trélat.
\newblock From microscopic to macroscopic scale equations: mean field,
  hydrodynamic and graph limits, 2022.

\bibitem{PiccoliPouradierDuteil21}
B.~Piccoli and N.~Pouradier~Duteil.
\newblock Control of collective dynamics with time-varying weights.
\newblock In F.~Salvarani, editor, {\em Recent Advances in Kinetic Equations
  and Applications}, pages 289--308, Cham, 2021. Springer International
  Publishing.

\bibitem{Popovych}
O.~Popovych, M.~Xenakis, and P.~Tass.
\newblock The spacing principle for unlearning abnormal neuronal synchrony.
\newblock {\em PLoS ONE}, 10, 02 2015.

\bibitem{porat2023graph}
I.~B. Porat, J.~A. Carrillo, and P.-E. Jabin.
\newblock The graph limit for a pairwise competition model, 2023.

\bibitem{PouradierDuteil21}
N.~Pouradier~Duteil.
\newblock Mean-field limit of collective dynamics with time-varying influence.
\newblock {\em Preprint}, 2020.

\bibitem{Rohr}
V.~Rohr, R.~Berner, E.~Lameu, O.~Popovych, and S.~Yanchuk.
\newblock Frequency cluster formation and slow oscillations in neural
  populations with plasticity.
\newblock {\em PLOS ONE}, 14:e0225094, 11 2019.

\bibitem{Rumelhart}
D.~E. Rumelhart, R.~J. Williams, and G.~E. Hinton.
\newblock Learning representations by back-propagating errors.
\newblock {\em Nature}, 323:533--536, 1986.

\bibitem{Serfaty20}
S.~Serfaty.
\newblock Mean field limit for coulomb-type flows.
\newblock 2020.

\bibitem{Sznitman91}
A.-S. Sznitman.
\newblock Topics in propagation of chaos.
\newblock {\em Lecture notes in mathematics}, pages 165--251, 1991.

\bibitem{throm2023continuum}
S.~Throm.
\newblock Continuum limit for interacting systems on adaptive networks, 2023.

\bibitem{Vicsek95}
T.~Vicsek, A.~Czir{\'o}k, E.~Ben-Jacob, I.~Cohen, and O.~Shochet.
\newblock Novel type of phase transition in a system of self-driven particles.
\newblock {\em Physical review letters}, 75(6):1226, 1995.

\bibitem{WileyStrogatzGirvan06}
D.~A. Wiley, S.~H. Strogatz, and M.~Girvan.
\newblock {The size of the sync basin}.
\newblock {\em Chaos: An Interdisciplinary Journal of Nonlinear Science},
  16(1), 03 2006.
\newblock 015103.

\bibitem{Zschaler_2012}
G.~Zschaler, G.~A. Bohme, M.~Seissinger, C.~Huepe, and T.~Gross.
\newblock Early fragmentation in the adaptive voter model on directed networks.
\newblock {\em Physical Review E}, 85(4), Apr. 2012.

\end{thebibliography}

\end{document}